\documentclass[a4paper,12pt]{article}

\usepackage[bbgreekl]{mathbbol}
\usepackage{mathrsfs}
\usepackage{graphicx}
\usepackage{amsmath}
\usepackage{amsfonts}
\usepackage{amssymb}
\usepackage{amsthm}
\usepackage{color}
\usepackage{cancel}
\usepackage{subfigure}
\usepackage{appendix}

\newtheorem{theorem}{Theorem}[section]
\newtheorem{lemma}{Lemma}[section]
\newtheorem{remark}{Remark}[section]
\newtheorem{proposition}{Proposition}[section]

\renewcommand{\theequation}{\arabic{section}.\arabic{equation}}
\renewcommand{\thetheorem}{\arabic{section}.\arabic{theorem}}
\renewcommand{\thelemma}{\arabic{section}.\arabic{lemma}}
\renewcommand{\theproposition}{\arabic{section}.\arabic{proposition}}

\renewcommand{\thealgorithm}{\arabic{section}.\arabic{algorithm}}
\renewcommand{\thefigure}{\arabic{section}.\arabic{figure}}

\newcommand{\sv}{\mathcal{S}^K_{NL}(\Omega)\cup H^1_{\#}(\Omega)}
\newcommand{\hdel}{H_{\#\delta}(\Omega)}
\newcommand{\hs}{\widetilde{H}^s(\Omega)}
\newcommand{\Ne}{N_{\rm e}}

\title{A discontinuous Galerkin scheme for full-potential electronic structure calculations
}
\author{Xiaoxu Li
\thanks{
School of Mathematical Sciences, Beijing Normal University, No.19 Xinjiekouwai Street, Beijing, 100875, P.R.China.
E-mail: {\tt xiaoxuli@mail.bnu.edu.cn}.
Xiaoxu Li's work was partially supported by the National Science Foundation for Young Scientists of China under Grant 11701037.
}
and Huajie Chen
\thanks{School of Mathematical Sciences, Beijing Normal University, No.19 Xinjiekouwai Street, Beijing, 100875, P.R.China.
E-mail: {\tt chen.huajie@bnu.edu.cn}.
Huajie Chen's work was partially supported by the Fundamental Research Funds for the Central Universities of China under Grant 2017EYT22.
}
}
\date{}

\begin{document}
\maketitle

\begin{abstract}
In this paper, we construct an efficient numerical scheme for full-potential electronic structure calculations of periodic systems. 
In this scheme, the computational domain is decomposed into a set of atomic spheres and an interstitial region, and different basis functions are used in different regions:
radial basis functions times spherical harmonics in the atomic spheres and plane waves in the interstitial region. 
These parts are then patched together by discontinuous Galerkin (DG) method.
Our scheme has the same philosophy as the widely used (L)APW methods in materials science, but possesses systematically spectral convergence rate.
We provide a rigorous {\em a priori} error analysis of the DG approximations for the linear eigenvalue problems,
and present some numerical simulations in electronic structure calculations.
\end{abstract}

\section{Introduction}\label{sec-intro} 
\setcounter{equation}{0}

Electronic structure calculations describe the energies and distributions of electrons, which plays a fundamental role in many different fields: 
materials science, biochemistry, solid-state physics, and surface physics. 
Among different electronic structure models, the Kohn-Sham density functional theory (DFT) \cite{martin05} 
so far achieves the best compromise between accuracy and computational cost.
For an $\Ne$-electron system with the presence of $M$ nuclei of charge $Z_k$ and located at ${\bf R}_k\in\mathbb{R}^3$~$(k=1,\cdots,M)$,
Kohn-Sham DFT gives rise to the following nonlinear eigenvalue problems
\begin{eqnarray}\label{eq:eigen}
H_{\Phi}\phi_i=\lambda_i\phi_i,\quad\lambda_1\leq\lambda_2\leq\cdots\leq\lambda_{\Ne},
\end{eqnarray}
with $\Phi=\{\phi_1,\cdots,\phi_{\Ne}\}$ and the Kohn-Sham Hamiltonian
\begin{eqnarray*}
H_{\Phi}=-\frac{1}{2}\Delta+V_{\rm ext}+V_{\rm H}[\rho_{\Phi}]+V_{\rm xc}[\rho_{\Phi}]. 
\end{eqnarray*}
Here, $\displaystyle V_{\rm ext}({\bf x})=-\sum_{k=1}^M\frac{Z_k}{|{\bf x}-{\bf R}_k|}$  is the external potential generated by nuclear attraction,
$\displaystyle V_{\rm H} [\rho_{\Phi}]=\int_{\mathbb{R}^3}\frac{\rho_{\Phi}({\bf y})}{|\cdot-{\bf y}|}d{\bf y}$ and
$ V_{\rm xc}[\rho_{\Phi}]$ are the so-called Hartree potential and exchange-correlation potential, respectively,
with the electron density $\displaystyle \rho_{\Phi}({\bf x})=\sum_{i=1}^{\Ne}|\phi_i({\bf x})|^2$.
A self-consistent field (SCF) iteration algorithm is commonly resorted to for these nonlinear problems. In each iteration, a
Hamiltonian $H_{\tilde{\Phi}}$ is constructed from a trial electronic state $\tilde{\Phi}$, and a linear eigenvalue problem is
then solved to obtain the low-lying eigenvalues and corresponding eigenfunctions. 
The loop continues until self-consistency of the electronic states is achieved.
The efficiency of the algorithm is mainly determined by the discretization of the Hamiltonian, the self-consistent iteration, and the
linear eigensolver. We shall focus ourselves on the discretization method in this paper.

For periodic systems, plane waves with {\it pseudopotentials} are natural methods 
which are simple to implement and give relatively accurate simulations. 
The pseudopotential approximations \cite{martin05} replace singular nuclear attraction potential and 
complicated effects of the motion of core electrons by a smooth potential. 
They give satisfactory results in most cases, but sometimes fail.
The mathematical analysis of the pseudopotential approximations is very rare, and we refer to \cite{blanc17,cances16} for two recent works.
Moreover, the core electrons have to be considered sometimes
and are responsible for many properties. Therefore, the full-potential/all-electron calculations are necessary.

For eigenvalue problems with singular potentials in full-potential calculations,
plane waves are inefficient bases for describing the cusps
at the nuclei positions \cite{fournais02,fournais04, fournais07,hoffmann01}.
In contrast, it is observed that a significant part of the rapid oscillations can be captured by atomic orbitals
such as Gaussians and Slater-type orbitals \cite{herring40,martin05}, which have been widely used in quantum chemistry 
(we refer to \cite{bachmayr14,chen15b} for their numerical analysis).
Therefore, it would be practically efficient to approximate the wavefunction in a crystal by using combinations of
plane waves and appropriate atomic orbitals. Several computational methods using this idea
have been developed, for example,
augmented plane waves (APW) \cite{singh06,slater37}, linearized augmented plane waves (LAPW) \cite{singh06},
and their extensions by including local orbitals (lo), LAPW+lo methods \cite{madsen01,schwarz02,sjostedt00}.
Exploiting the idea of constructing basis functions for different domains separately,  we construct a numerical scheme in this paper.
The smoothly varying parts of the wavefunctions away from the atoms are represented by plane waves,
the rapidly varying parts near the nuclei are represented by radial basis functions times spherical harmonics, 
and the approximations inside and outside the spheres are patched together by DG methods.

The DG framework has been widely used in numerical solutions of partial differential equations and investigated
theoretically in a lot of works (see, e.g., \cite{arnold00,babuska73,buffa08,harriman03,suli00} and references cited therein).
For electronic structure calculations,  we refer to works by Lin et al. \cite{lin12, zhang17},
which constructs basis functions adaptively from the local environment and patches them together in global domain by DG methods.

We further present an {\em a priori} error analysis of our DG approximations for the linear eigenvalue problems.
Thanks to the asymptotic regularity result developed
by Flad et al. \cite{flad08}, we can guarantee smoothness of the wavefunctions on the  domain $[0,R]\times S^2$ in spherical coordinates.
Our analysis for DG approximations is also closely related to the technique used in \cite{antonietti06,harriman03,osborn75,suli00}.
The main theoretical result in this paper is the following superalgebraic convergence rate under certain assumptions
(see Theorem \ref{theo-approximation-rate}):
\begin{eqnarray*}
|\lambda-\lambda_i^{\rm DG}| + \|u_i-u_i^{\rm DG}\|_{\rm DG} \leq  C_s\varrho^{\frac{3}{2}+\epsilon-s}
\qquad \forall~s\in\mathbb{R}^+ ,
\end{eqnarray*}
where $\epsilon>0$ can be arbitrarily small, $\varrho$ denotes the discretization parameters (see \eqref{add-assumption-varrho}), 
and the constant $C_s$ depends only on $s$ and the eigenfunctions.

We shall briefly compare our DG method with other existing full-potential/all-electron methods in electronic structure calculations.
(a)
{\it APW}:
The augmented plane wave (APW) method \cite{slater37} introduces basis functions 
that are plane waves in the interstitial region and radial solutions of Schr\"{o}dinger equations inside the atomic spheres.
A great disadvantage of the APW method is that the basis functions are energy dependent, which results in a
nonlinear eigenvalue problem and must be solved separately for each eigenstate by ``root tracing" technique \cite{martin05} or iteration methods. 
This method is much more complicated to solve than the straightforward linear eigenvalue equations expressed with a fixed basis set, such as plane waves, Gaussians, LAPW (in the following), and our DG schemes.
(b)
{\it LAPW (+lo)}:
The linearized augmented plane wave (LAPW) method \cite{martin05,singh06} is a  linearization of APW, which defines basis functions as 
linear combinations of a radial solution and its energy derivative evaluated at a chosen fixed energy. 
This forms a basis set  adapted to a particular system that is suitable for calculation of all states in an energy ``window". 
The accuracy depends heavily on the choice of energy parameter and the width of the energy window under consideration.
Although the inclusion of additional variational freedoms (the energy derivatives and sometimes local orbitals (lo)
\cite{singh06}) in the LAPW method facilitates the computation for non-spherical symmetric parts
of the potential, there is no proof that it can give solutions of arbitrarily great accuracy for general potentials as our DG scheme.
Here we would like to mention a recent work \cite{drescher14} which uses similar ideas as LAPW+lo and may possess systematical convergence.
(c)
{\it OPW}:
The orthogonalized plane wave (OPW) method \cite{herring40} constructs basis functions by orthogonalizing the plane waves to special local functions around each
nucleus. The ambiguity of this method arises from inaccuracies of the core wave functions, which are not precise eigenfunctions of the given Hamiltonian. Thus, there is always an
uncertainty about the accuracy of OPW results which can not be refined out by more extended calculations.
(d)
{\it PAW/VPAW}:
The projector augmented wave (PAW) method \cite{bloch94} replaces the original eigenvalue problem (with singular potential) by a new one 
with the same eigenvalues but smoother eigenvectors. 
A slightly different method, called variational projector augmented wave (VPAW), was proposed and analyzed recently \cite{blanc17}. 
This new method allows for a better convergence with respect to the number of plane waves. 
But we mention that the PAW method is more of a pseudopotential method. 

The rest part of this paper is organized as follows. 
In Section \ref{sec-problem}, we set up the model problem and present some regularity results.
In Section \ref{sec-discretization}, we introduce a DG discretization scheme,
provide a numerical analysis of the convergence and {\it a priori} error estimates of the DG approximations.
In Section \ref{sec-numerical}, we give some details of the numerical implementations and present some numerical experiments to support our theory.
Finally, we give some concluding remarks.

\section{Preliminary}\label{sec-problem} \setcounter{equation}{0}

Throughout this paper, we shall use $C$ to denote a generic positive constant which may
stand for different values at its different occurrences and is independent of finite dimensional subspaces.
For convenience, the symbol $\lesssim$ will be used and the notation $A\lesssim B$ means that $A\leq CB$ for some generic positive constant $C$.

Let $\mathcal{R}$ be a discrete periodic lattice of $\mathbb{R}^3$, $\Omega$ be the unit cell of the lattice, and $\mathcal{R}^*$ be the dual lattice.
For simplicity, we take $\Omega=[-\frac{D}{2},\frac{D}{2}]^3~(D>0)$, 
$\mathcal{R} = D\mathbb{Z}^3$, and $\mathcal{R}^*=\frac{2\pi}{D}\mathbb{Z}^3$.

For ${\bf k}\in\mathcal{R}^*$, we denote by $e_{\bf k}({\bf r})=|\Omega|^{-1/2} e^{i{\bf k\cdot r}}$ the plane wave
with wavevector ${\bf k}$. The family $\{e_{\bf k}\}_{{\bf k}\in\mathcal{R}^*}$ forms an orthonormal basis set of
\begin{eqnarray*}
L_{\#}^2(\Omega)=\{u\in L^2_{\rm loc}(\mathbb{R}^3)~:~u~{\rm is}~\mathcal{R}{\rm -periodic}\}.
\end{eqnarray*}
For all $u\in L_{\#}^2(\Omega)$, we have
\begin{eqnarray*}
u({\bf r})=\sum_{{\bf k}\in\mathcal{R}^*}\hat{u}_{\bf k} e_{\bf k}({\bf r})\quad{\rm with}\quad \hat{u}_{\bf k}
=(u,e_{\bf k})_{L_{\#}^2(\Omega)}=|\Omega|^{-1/2}\int_{\Omega} u({\bf r})e^{-i{\bf k\cdot r}}d{\bf r}.
\end{eqnarray*}
We introduce the Sobolev spaces of $\mathcal{R}$-periodic functions
\begin{eqnarray*}
H_{\#}^s(\Omega)=\left\{u({\bf r})=\sum_{{\bf k}\in\mathcal{R}^*}\hat{u}_{\bf k} e_{\bf k}({\bf r})~:~
\sum_{{\bf k}\in\mathcal{R}^*}\big(1+|{\bf k}|^2\big)^s |\hat{u}_{\bf k}|^2<\infty
\right\},
\end{eqnarray*}
with $s\in\mathbb{R}^+$.
For $K\in \mathbb{N}^+$, we denote the finite dimensional subspace by
\begin{eqnarray*}
\mathcal{V}_K=\left\{v_K({\bf r})=\sum_{{\bf k}\in\mathcal{R}^*,|{\bf k}|\leq \frac{2\pi}{D}K} c_{\bf k}e_{\bf k}({\bf r}) \right\}.
\end{eqnarray*}
For $v\in H^s_{\#}(\Omega)$, the best approximation of $v$ in $\mathcal{V}_K$ is
$\Pi_K v=\sum_{{\bf k}\in\mathcal{R}^*,|{\bf k}|\leq \frac{2\pi}{D}K} \hat{v}_{\bf k} e_{\bf k}({\bf r})$ for any $H^t$-norm ($t\leq s$).
The more regularity $v$ has,
the faster this truncated series converge to $v$: For real numbers $t$ and $s$ satisfying $t\leq s$, we have that for each
$v\in H^s_{\#}(\Omega)$,
\begin{eqnarray}\label{add-rate-planewave}
\|v-\Pi_K v\|_{H^t_{\#}(\Omega)}=\min_{v_K\in \mathcal{V}_K}\|v-v_K\|_{H^t_{\#}(\Omega)}\lesssim K^{t-s}\|v\|_{H^s_{\#}(\Omega)}.
\end{eqnarray}

As a model problem, we  consider the following Schr\"{o}dinger-type linear eigenvalue problem, which can be viewed as a
linearization of \eqref{eq:eigen}: Find $\lambda\in\mathbb{R}$ and $0\neq u\in H_{\#}^1(\Omega)$ such that $\|u\|_{L_{\#}^2(\Omega)=1}$ and
\begin{eqnarray}\label{eq-model}
a(u,v)=\lambda(u,v) \qquad\forall~v\in H_{\#}^1(\Omega),
\end{eqnarray}
where the bilinear form $a(\cdot,\cdot):H_{\#}^1(\Omega)\times H_{\#}^1(\Omega)\rightarrow\mathbb{C}$ is given by
\begin{eqnarray}\label{bilinear-a}
a(u,v)=\frac{1}{2}\int_{\Omega}\nabla u\cdot\nabla v+\int_{\Omega} Vuv 
\end{eqnarray}
with  a $\mathcal{R}$-periodic potential $V\in L^2_{\#}(\Omega)$.

To represent the wavefunctions separately in different regions, $\Omega$ is divided into atomic spheres 
and an interstitial region (see Figure \ref{fig-division} (left) for decomposition of a single-atom system,
and Figure \ref{fig-division} (right) for similar construction of a two-atom system).

For sake of simplicity, we shall restrict our discussions to a single atom located at the origin,
the algorithms and analysis of which can be easily generalized to multi-atom systems.
Throughout this paper, we shall denote by $\Omega_{\rm out}$ the interstitial region, by
$\Omega_{\rm in}$ the sphere centered at the origin with radius $R$, and by $\Gamma$ the spherical surface.
We also assume throughout this paper that the potential $V$ equals to $-Z/|{\bf r}|$ in the neighborhood of ${\bf 0}$,
and belongs to $C_{\rm loc}^{\infty}(\mathbb{R}^3\setminus\mathcal{R})\cap L^2_{\#}(\Omega)$.

\begin{figure}[htbp]
	\centering
	\subfigure{
	\begin{minipage}{6.0cm}
		\centering
		\includegraphics[width=6.0cm]{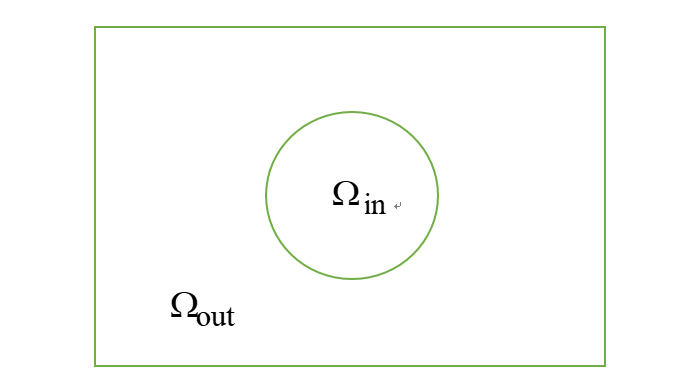}
	\end{minipage}
	}
	\subfigure{
		\begin{minipage}{6.0cm}
			\centering
			\includegraphics[width=6.0cm]{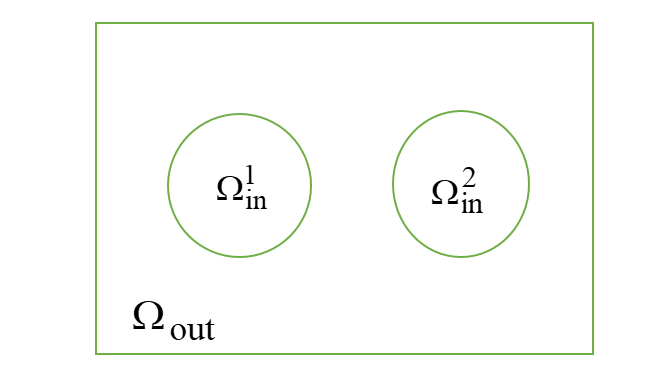}
		\end{minipage}
	}
	\caption{The  division of $\Omega$ into atomic spheres $\Omega_{\rm in}$ and a interstitial region $\Omega_{\rm out}$.}
	\label{fig-division}
\end{figure}

It was shown in \cite{fournais02,fournais04,fournais07} that the exact electron densities are analytic away from the nuclei and satisfy certain cusp conditions at the nuclei. 
The plane wave approximations can not have as good convergence rate
as \eqref{add-rate-planewave} due to the cusps at the nuclear positions.
The following lemma concerning the regularity of eigenfunctions of \eqref{eq-model} is heavily used in our analysis,
the proof of which can be referred to \cite[Theorem 1,4 and Proposition 1]{flad08}.

\begin{lemma}\label{lemma-regularity}
If $u$ is an eigenfunction of \eqref{eq-model}, then $ u\in H^s([0,R]\times S^2) $ for any $s\in\mathbb{Z}^+$.
\end{lemma}

The following lemma states the relationship between two Sobolev norms.

\begin{lemma}\label{lemma-H1norm}
If $v\in H^1(\Omega_{\rm in})\bigcap H^3([0,R]\times S^2)$, then there exits a constant $C_R$ depending on $R$ such that
\begin{eqnarray*}
\|v\|_{H^1(\Omega_{\rm in})}\leq C_R\|v\|_{H^1([0,R]\times S^2)}.
\end{eqnarray*}
\end{lemma}
\begin{proof}
Note that in spherical coordinates
\begin{eqnarray*}
\Delta=\frac{1}{r^2}\frac{\partial}{\partial r}\left( r^2\frac{\partial}{\partial r} \right)
+\frac{1}{r^2\sin\theta}\frac{\partial}{\partial\theta}\left( \sin\theta\frac{\partial}{\partial\theta} \right)
+\frac{1}{r^2\sin^2\theta}\frac{\partial^2}{\partial^2\phi},
\end{eqnarray*}
where the last two terms multiplied by $r^2$ is the total angular momentum operator $\Delta_{S^2}$ on spherical
surface, i.e. the Laplace-Beltrami operator. 

Since $v\in H^1(\Omega_{\rm in})\bigcap H^3([0,R]\times S^2)$ implies
\begin{eqnarray}\label{proof-0-1}
\lim_{r\rightarrow 0} r^2 v\frac{\partial v}{\partial r}=0,
\end{eqnarray}
we have
\begin{eqnarray*}
\|v\|^2_{H^1(\Omega_{\rm in})} &=& -\int_{\Omega_{\rm in}}v\Delta v+\int_{\Gamma}v\left.\frac{\partial v}{\partial r}
\right|_{r=R}+\int_{\Omega_{\rm in}}v^2 \\
&=& \int_0^R r^2dr 
\int_{S^2} \left(v^2+(\frac{\partial v}{\partial r})^2 \right)
-\int_0^R dr\int_{S^2}(v\Delta_{S^2}v)\\ \nonumber
&\leq& R^2\int_{S^2}\|v\|^2_{H^1([0,R])} + \int_0^R \|v\|^2_{H^1(S^2)}dr\\
&\leq& C_{R}\|v\|^2_{H^1([0,R]\times S^2)},
\end{eqnarray*}
where Green's formula and \eqref{proof-0-1} are used for the second equality.
\end{proof}

\section{DG discretization}\label{sec-discretization} \setcounter{equation}{0}

In this section, we construct a DG discretization scheme using radial basis functions times spherical harmonics inside the sphere
and plane waves outside. 
We provide an {\it a priori} error analysis of the numerical approximations.
Our analysis is composed of three steps: 
first, we estimate the best approximation errors inside and outside the sphere separately;
then we give an error estimate for the DG approximation of the corresponding source problem;
finally, we derive an error estimate for the eigenvalue problem.
Note that the errors generated by numerical quadratures and linear algebraic
solvers are not considered in this paper, which deserve separate investigations.

Note that if $u$ is an eigenfunction of \eqref{eq-model},  then we have from Lemma \ref{lemma-regularity} that for any $s>0$, 
$u|_{\Omega_{\rm in}}\in H^s([0,R]\times S^2) $ in spherical coordinates and $u|_{\Omega_{\rm out}}\in H^s(\Omega_{\rm out})$. 
We can therefore introduce the following space 
\begin{eqnarray*}
	\widetilde{H}^s(\Omega)=\left\{v\in H^1_{\#}(\Omega):v|_{\Omega_{\rm in}}\in H^s([0,R]\times S^2),v|_{\Omega_{\rm out}}\in H^s(\Omega_{\rm out})\right\}
\end{eqnarray*}
with induced norm 
\begin{eqnarray*}
\|v\|_{\hs}:=\|v\|_{H^s(\Omega_{\rm out})} + \|v\|_{H^s([0,R]\times S^2)}.
\end{eqnarray*}

\subsection{Approximation space}\label{subsec-space}

Denote by $\mathcal{P}_K(\Omega_{\rm out})$ the space of functions on $\Omega_{\rm out}$ expanded by plane waves
\begin{eqnarray*}
\mathcal{P}_K(\Omega_{\rm out})=\left\{u\in H^1(\Omega_{\rm out}),~u({\bf r})
=\sum_{|{\bf k}|\leq \frac{2\pi}{D}K} c_{\bf k}e_{\bf k}({\bf r}) \bigg|_{\Omega_{\rm out}} \right\}	
\end{eqnarray*}
and by $\mathcal{B}_{NL}$ the space of functions on $\Omega_{\rm in}$ expanded by radial basis functions times spherical harmonics
\begin{eqnarray*}
\mathcal{B}_{NL}(\Omega_{\rm in})=\left\{u\in H^1(\Omega_{\rm in}),~u({\bf r})=\check{u}(r,\theta,\phi)
=\sum_{0\leq n\leq N,0\leq l\leq L,|m|\leq l} c_{nlm}\chi_n(r)Y_{lm}(\theta,\phi) \bigg|_{\Omega_{\rm in}} \right\},
\end{eqnarray*}
where $\{\chi_n\}_{n=0}^N$ are basis functions on $[0,R]$. 
Here, we denote by $\check{u}(r,\theta,\phi)$ the spherical coordinate representations of the function $u({\bf r})$,
i.e., $u({\bf r})=\check{u}(r,\theta,\phi)$.

For simplicity, we may assume that the radial basis functions $\{\chi_n\}$  are polynomials that span the
space of all polynomials of degree no greater than $N$.

\begin{remark}
There are many choices of the radial basis functions $\{\chi_n\}$.
One can use spectral methods,
such as Legendre polynomials, Chebyshev polynomials and Jacobi polynomials, et al. 
These types of functions form a complete basis set on $[0,R]$, and possess spectral convergence rates for any sufficiently smooth function \cite{canuto88}.
	
Another type of basis set is atomic orbitals \cite{lebris03,martin05}, such as Gaussians, Slater-type orbitals
and numerical solutions of radial Schr\"{o}dinger equations \cite{lebris03,martin05}.
These basis functions are closely related to physical problems, and can be very efficient in practice. 
In some cases, we can rigorously derive their convergence rates \cite{bachmayr14, chen15b}.
	
For simplicity of presentations, we shall focus our analysis on the polynomial-type radial basis functions and investigate the atomic orbitals as well in numerical experiments.
\end{remark}

Define the finite dimensional space
\begin{align*}
\mathcal{S}^K_{NL}(\Omega) :&=
\mathcal{P}_K(\Omega_{\rm out})\oplus\mathcal{B}_{NL}(\Omega_{\rm in})
\\
&= \left\{ u\in L^2_{\#}(\Omega),~u|_{\Omega_{\rm in}}\in\mathcal{B}_{NL}(\Omega_{\rm in})
\mbox{ and } u|_{\Omega_{\rm out}}\in\mathcal{P}_K(\Omega_{\rm out}) \right\}.
\end{align*}
Throughout this paper, we may assume that there exists a constant $\varrho$ such that
\begin{eqnarray}\label{add-assumption-varrho}
\max\{K,N,L\} \leq \varrho \leq C\min\{K,N,L\},
\end{eqnarray}
which can denote the discretization parameters.

In the following, we shall define some ``best" approximations of the function in the interstitial region and atomic spheres respectively.

For the interstitial region, 
we define the projections
$P_K:L^2_{\#}(\Omega_{\rm out})\rightarrow \mathcal{P}_K(\Omega_{\rm out})$ satisfying
\begin{eqnarray*}
\|u-P_K u\|_{H^1(\Omega_{\rm out})}=\inf_{U^{\rm out}_K\in \mathcal{P}_K(\Omega_{\rm out})}\|u-U^{\rm out}_K\|_{H^1(\Omega_{\rm out})} .
\end{eqnarray*}
%

\begin{proposition}\label{propositon-planewave}
	If $u\in H^s(\Omega_{\rm out}) $, then for  $0\le t<s$, there exists a constant $C$ such that
	\begin{eqnarray}\label{rate-planewave}
	\|u-P_K u\|_{H^t(\Omega_{\rm out})}\leq CK^{t-s}\|u\|_{H^s(\Omega_{\rm out})}.
	\end{eqnarray}
\end{proposition}

\begin{proof}
	The proof is similar to that of \cite[Proof of Lemma 3.1]{chen15a}.
	We keep this proof for sake of completeness.
	
	We shall first extend the function $u|_{\Omega_{\rm out}}$ smoothly into the sphere.
	The wavefunction around the sphere can be represented by
	\begin{eqnarray*}\label{eq-wavefunction}
		u({\bf r})=\sum_{lm}^{\infty}u_{lm}(r)\tilde{Y}_{lm}({\bf r})
	\end{eqnarray*}
	with $\displaystyle u_{lm}(r)=\int_0^{\pi}\sin\theta\int_0^{2\pi}u(r,\theta,\phi)Y_{lm}(\theta,\phi) d\phi d\theta$, where we use spherical coordinates ${\bf r}\rightarrow (r,\theta,\phi)$
	to express $\tilde{Y}_{lm}({\bf r})=Y_{lm}(\theta,\phi)$ for spherical harmonics on $S^2$.
	Then we can define
	\begin{eqnarray}\label{eq-extension-phi}
	\tilde{u}({\bf r})=\left\{ \begin{array}{ll} u({\bf r}) & \mbox{in }\Omega_{\rm out}, \\[1ex]\displaystyle
	\sum_{lm}^{\infty}\varphi_{lm}(r)\tilde{Y}_{lm}({\bf r}) & \mbox{in }\Omega_{\rm in}, \end{array} \right.
	\end{eqnarray}
	where $\displaystyle \varphi_{lm}(r)=\tau(r)\sum_{n=1}^{s+1}c_n u_{lm}(R+\frac{1}{n}(R-r))$ with
	$\displaystyle \sum_{n=1}^{s+1}(-\frac{1}{n})^k c_n=1~(k=0,1,\cdots,s)$,
	$\tau\in C^{\infty}([0,R])$ satisfying $\tau=0$ on $[0,\frac{R}{3}]$ and $\tau=1$ on $[\frac{2R}{3},R]$.
	We observe that $u\in H^s(\Omega_{\rm out})$ leads to $\tilde{u}\in H^s(\Omega)$ and moreover
	\begin{eqnarray}\label{eq-extension-norm}
	\|\tilde{u}\|_{H^s(\Omega)} \leq\beta\|u\|_{H^s(\Omega_{\rm out})},
	\end{eqnarray}
	where the constant $\beta$ is only related to $s,~R$ and $\|\tau\|_{C^{\infty}([0,R])}$.
	Let 
	\begin{align*}
	\tilde{u}_K=\sum_{|{\bf k}|\leq \frac{2\pi}{D}K }\tilde{c}_{\bf k}e_{\bf k}
	\qquad {\rm with} \qquad 
	\tilde{c}_{\bf k} =\int_{\Omega}\tilde{u}({\bf r})e^{*}_{\bf k}({\bf r})d{\bf r} ,
	\end{align*}
	we have from \eqref{eq-extension-norm} that
	\begin{align*}
	\inf_{U^{\rm out}_K\in\mathcal{P}_K(\Omega_{\rm out})}\|u-U^{\rm out}_K\|_{H^t(\Omega_{\rm out})} 
	&\leq \|u-\tilde{u}_K\|_{H^t(\Omega_{\rm out})}
	~\leq~\|\tilde{u}-\tilde{u}_K\|_{H^t(\Omega)} \\
	&\leq CK^{t-s}\|\tilde{u}\|_{H^s(\Omega)} ~\leq~C\beta K^{t-s}\|u\|_{H^s(\Omega_{\rm out})},
	\end{align*}
	which completes the proof of \eqref{rate-planewave}.
\end{proof}


For the atomic spheres, we
define
$P_N:H^1([0,R])\rightarrow\Psi_N \equiv {\rm span}\{\chi_n\}_{n=1}^{N} $ satisfying
\begin{eqnarray*}
\|v-P_N v\|_{H^1([0,R])}=\inf_{\psi_N\in\Psi_N}\|v-\psi_N\|_{H^1([0,R])},
\end{eqnarray*}
and $P_L:L^2(S^2)\rightarrow \mathcal{Y}_L \equiv {\rm span}\{Y_{lm},0\leq l\leq L,-l\leq m\leq l\}$ satisfying
\begin{eqnarray*}
P_L\varphi(\theta,\phi)=\sum_{l=0}^L\sum_{m=-l}^l \hat{\varphi}_{lm}Y_{lm}(\theta,\phi) \quad{\rm with}\quad
\hat{\varphi}_{lm}=\int_0^{\pi}\sin\theta \int_0^{2\pi}\varphi(\theta,\phi)Y^*_{lm}(\theta,\phi) d\phi d\theta.
\end{eqnarray*}
For $P_N$ and $P_L$, we have the following standard estimates (see, e.g., \cite{atkinson09, maday81})
\begin{gather*}
\|v-P_N v\|_{H^t([0,R])} \leq CN^{t-s}\|v\|_{H^s([0,R])}, 
\\[1ex]
\|\varphi-P_L\varphi\|_{H^t(S^2)} \leq CL^{t-s}\|\varphi\|_{H^s(S^2)}
\end{gather*}
for any $0\le t\le 1$ and $t<s$.
Define the projection $P_{NL}:H^1([0,R]\times S^2)\rightarrow \Psi_N\times\mathcal{Y}_L$ by $P_{NL}=P_N\circ P_L$,
we have that for $w\in H^s([0,R]\times S^2)$, $0\le t\le 1$ and $t<s$,
\begin{eqnarray}\label{rate-P}
\|w-P_{NL}w\|_{H^t([0,R]\times S^2)}\leq C(L^{t-s}+N^{t-s})\|w\|_{H^s([0,R]\times S^2)}.
\end{eqnarray}

\begin{proposition}\label{proposition-radial}
If $u\in H^s([0,R]\times S^2)\bigcap H^1(\Omega_{\rm in})$, then for $0\le t\le 1$  and any $s\geq3$, 
there exists a constant $C$ such that
\begin{eqnarray}\label{rate-sphere}
\|u-P_{NL}u\|_{H^t(\Omega_{\rm in})}  \leq C
(L^{t-s}+N^{t-s})\|u\|_{H^s([0,R]\times S^2)}.
\end{eqnarray}
\end{proposition}

\begin{proof}	
Using \eqref{rate-P} and Lemma \ref{lemma-H1norm}, we have
\begin{eqnarray*}
\|u-P_{NL}u\|_{H^t(\Omega_{\rm in})} \leq
\|u-P_{NL}u\|_{H^t([0,R]\times S^2)} \leq
C(L^{t-s}+N^{t-s})\|u\|_{H^s([0,R]\times S^2)},
\end{eqnarray*}
which completes the proof.
\end{proof}

\subsection{DG approximations of the source problem}
\label{subsec-bvP}

We shall discuss the DG discretization for the source problem and our analysis is related to the framework in \cite{antonietti06}.

For vector-valued ${\bf w}$ and scalar-valued function $u$ which are not continuous on
the spherical surface $\Gamma$, we define the jumps by
\begin{eqnarray*}
[{\bf w}]={\bf w}^+ \cdot {\bf n}^+ +{\bf w}^- \cdot {\bf n}^-,\quad\quad [u]=u^+{\bf n}^+
+u^-{\bf n}^-
\end{eqnarray*}
and the averages by
\begin{eqnarray*}
\{{\bf w}\}=\frac{1}{2}({\bf w}^+ +{\bf w}^-),\quad\quad \{u\}=\frac{1}{2}(u^++u^-),
\end{eqnarray*}
where ${\bf w}^{\pm}$ and $u^{\pm}$ are traces of ${\bf w}$ and $u$ on $\Gamma$ taken from inside
and outside the sphere, ${\bf n}^{\pm}$ are the normal unit vectors.

Since $\mathcal{S}^K_{NL}(\Omega)$ is a finite dimensional space, 
there exists a constant $\gamma_\varrho$ depending on $\varrho$ such that
the following inverse estimate holds 
\begin{eqnarray}\label{finite-dimensional-norm}
\|u^{+}-u^{-}\|_{H^\frac{1}{2}(\Gamma)}\leq \gamma_\varrho\|u^{+}-u^{-}\|_{L^2(\Gamma)}\qquad\forall u\in\mathcal{S}^K_{NL}(\Omega). 
\end{eqnarray}
In our analysis, we assume 
\begin{eqnarray}\label{eq-inverse-H1-assumption}
\|u^{+}-u^{-}\|_{H^1(\Gamma)}\lesssim \varrho^{2}\|u^{+}-u^{-}\|_{L^2(\Gamma)}\quad\forall u\in\mathcal{S}^K_{NL}(\Omega).
\end{eqnarray}
We are not able to justify \eqref{eq-inverse-H1-assumption} rigorously, 
however, we provide some numerical experiments in \ref{sec:inverse-estimate} to show that it could be true.
Then we get from the ``interpolation" arguments (see \ref{sec:inverse-estimate}) that
\begin{eqnarray}\label{inverse-estimate}
\gamma_\varrho = C_{\epsilon} \varrho^{1+\epsilon} 
\end{eqnarray} 
for some $\epsilon\in(0,1)$.
%

We then define the bilinear form $a^{\rm DG}(\cdot,\cdot):\left(\sv\right) \times \left(\sv\right) \rightarrow\mathbb{C}$ by
\begin{eqnarray}\label{eq-bilinear-DG}
\nonumber
a^{\rm DG}(u,v) &=& \int_{\Omega_{\rm in}}\left(\frac{1}{2}\nabla u\cdot\nabla v+Vuv\right)
+\int_{\Omega_{\rm out}}\left(\frac{1}{2}\nabla u\cdot\nabla v+Vuv\right)\\[1ex]
&& -\frac{1}{2}\int_{\Gamma}\{\nabla u\}\cdot[v]ds - \frac{1}{2}\int_{\Gamma}\{\nabla v\}\cdot[u]ds
+ \int_{\Gamma}\sigma[u]\cdot[v]ds,
\end{eqnarray}
where $\sigma=C_{\sigma}\varrho^{2+2\epsilon}$ is the discontinuity-penalization parameter with a constant $C_{\sigma}$ independent of the discretization.

Note that there are many other types of DG formulations (see, e.g., \cite{antonietti06,arnold00}), 
and \eqref{eq-bilinear-DG} is the classical symmetric interior penalty (SIP) method \cite{antonietti06,arnold82,lin12}.

We further define the broken Sobolev space
\begin{align*}
\hdel = \left\{v\in L^2_{\#}(\Omega):
~v|_{\Omega_{\rm in}}\in H^1(\Omega_{\rm in}),~v|_{\Omega_{\rm out}}\in H^1(\Omega_{\rm out})\right\} 
\end{align*}
equipped with the following DG-norm
\begin{eqnarray}\label{eq-norm-DG}
\|u\|^2_{\rm DG}=\|u\|^2_{H^1(\Omega_{\rm in})}+\|u\|^2_{H^1(\Omega_{\rm out})}+\sigma\|[u]\|^2_{L^2(\Gamma)}.
\end{eqnarray}

\begin{lemma}
If $C_{\sigma}$ is sufficiently large, then there exist constants $\alpha,\beta>0$ such that
\begin{eqnarray}\label{coercive}
a^{\rm DG}(u,u)\geq \alpha\|u\|^2_{\rm DG}-\beta\|u\|^2_{L_{\#}^2(\Omega)}
\qquad \forall~u\in \sv .
\end{eqnarray}
\end{lemma}

\begin{proof}
Using H\"{o}lder inequality, Sobolev's embedding theorem and Young's inequality, we obtain that
\begin{multline*} 
	\qquad
	\left|\int Vu^2\right| 
	\leq \|V\|_{L^2}\cdot\|u^{\frac{1}{2}}\|_{L^4}\cdot \|u^{\frac{3}{2}}\|_{L^4}
	= C\|u\|^{\frac{1}{2}}_{L^2}\cdot\|u\|^{\frac{3}{2}}_{L^6}
	\\
	\leq C\|u\|^{\frac{1}{2}}_{L^2}\cdot\|u\|^{\frac{3}{2}}_{H^1}
	\leq C(\frac{\delta^{-4}}{4}\|u\|^2_{L^2}+\frac{3\delta^{\frac{3}{4}}}{4}\|u\|^{2}_{H^1}),
	\qquad
\end{multline*}
where $\delta>0$ is arbitrarily small.
Hence we have 
\begin{eqnarray}\label{coercive-veff}
\int Vu^2\geq -C\delta^{\frac{3}{4}}\|u\|^{2}_{H^1}-b\delta^{-4}\|u\|^2_{L^2}
\end{eqnarray}
with constants $C,~b>0$.
Moreover, we have
\begin{eqnarray}\label{coercive-jump}\nonumber
\left|\int_{\Gamma}\{\nabla u\}\cdot[u]ds\right|&\leq&\|\{\nabla u\}\cdot{\bf n}^{+}\|_{H^{-\frac{1}{2}}(\Gamma)}\cdot\|u^{+}-u^{-}\|_{H^{\frac{1}{2}}(\Gamma)}\\[1ex]\nonumber
&\leq&\delta^2 \|\{\nabla u\}\|^2_{H^{-\frac{1}{2}}(\Gamma)}+\delta^{-2}\|u^{+}-u^{-}\|^2_{H^{\frac{1}{2}}(\Gamma)}\\[1ex]
&\leq&C\delta^2(\|u\|^2_{H^1(\Omega_{\rm out})} + \|u\|^2_{H^1(\Omega_{\rm in})})+\delta^{-2}\|u^{+}-u^{-}\|^2_{H^{\frac{1}{2}}(\Gamma)}.
\end{eqnarray}
Using \eqref{finite-dimensional-norm}, \eqref{coercive-veff} and \eqref{coercive-jump}, we can derive \eqref{coercive} and complete the proof. 
\end{proof}

For simplicity, we can take $\beta=0$. Note that $a^{\rm DG}_{\beta}(u,v) =a^{\rm DG}(u,v)+\beta(u,v)$ makes this true for $\beta>0$.

Define the solution operators 
\begin{eqnarray*}
T:L_{\#}^2(\Omega)\rightarrow H_{\#}^1(\Omega) \quad a(Tf,v)=(f,v)\quad\forall~v\in H_{\#}^1(\Omega),
\end{eqnarray*}
and
\begin{eqnarray*}
T^{\rm DG}:L_{\#}^2(\Omega)\rightarrow \mathcal{S}^K_{NL}(\Omega)
\quad a^{\rm DG}(T^{\rm DG}f,v)=(f,v)\quad\forall~v\in \mathcal{S}^K_{NL}(\Omega).
\end{eqnarray*}

\begin{proposition}\label{proposition-T-approximate}
Assume that \eqref{eq-inverse-H1-assumption} is true and  $C_{\sigma}$ is sufficiently large. 
If $Tf\in H^s(\Omega_{\rm out})\oplus H^s([0,R]\times S^2)$
for $f\in L^2_{\#}(\Omega)$ and $s\geq3$,
then there exists a constant $C$ such that
\begin{eqnarray}\label{rate-T-tildeT}
\|(T-T^{\rm DG})f\|_{\rm DG} \leq C\varrho^{\frac{3}{2}+\epsilon-s}
\|Tf\|_{\hs}.
\end{eqnarray}
\end{proposition}

\begin{proof}
Denote $w=Tf$ and $w^{\rm DG}=T^{\rm DG}f$.
Define the projection $\mathcal{P}u=P_{K}u|_{\Omega_{\rm out}}+P_{NL}u|_{\Omega_{\rm in}}$,
We decompose the error $e=w-w^{\rm DG}$ as $e=\eta+\xi$, where
$\eta=w-\mathcal{P}w$ and $\xi=\mathcal{P}w-w^{\rm DG}$.
With simple calculations, we can easily obtain that $a^{\rm DG}(w,\xi)=(f,\xi)$, which leads to the property that $a^{\rm DG}(w-w^{\rm DG},\xi)=0$.
Using \eqref{coercive} and the property, we have
\begin{eqnarray*}
\|\xi\|^2_{\rm DG}\lesssim a^{\rm DG}(\xi,\xi)=a^{\rm DG}(e-\eta,\xi)=-a^{\rm DG}(\eta,\xi).
\end{eqnarray*}
Thus we deduce that
\begin{eqnarray}\label{proof-r0}
\|\xi\|^2_{\rm DG} \lesssim I_1+I_2+I_3,
\end{eqnarray}
where
\begin{align*}
&I_1 = \left|\int_{\Omega_{\rm in}}(\frac{1}{2}\nabla\eta\cdot\nabla\xi+V\eta\xi)\right|
+\left|\int_{\Omega_{\rm out}}(\frac{1}{2}\nabla\eta\cdot\nabla\xi+V\eta\xi)\right|, \\[1ex]
&I_2 = \frac{1}{2}\left|\int_{\Gamma}\{\nabla\eta\}\cdot[\xi]ds+\int_{\Gamma}\{\nabla\xi\}\cdot[\eta]ds\right|, \\[1ex]
&I_3 = \left|\int_{\Gamma}\sigma[\eta]\cdot[\xi]ds\right|.
\end{align*}
	
Since $V\in L_{\#}^2(\Omega)$, we have
\begin{eqnarray}\label{proof-r1}
I_1 \lesssim \|\xi\|_{\rm DG}(\|\eta\|_{H^1(\Omega_{\rm in})}+\|\eta\|_{H^1(\Omega_{\rm out})}).
\end{eqnarray}
Using the trace inequality, $I_2$ can be estimated by
\begin{align}\label{proof-r2} 
\nonumber
& \hskip -0.2cm
I_2 \lesssim \|[\xi]\|_{H^{\frac{1}{2}}(\Gamma)}\|\{\nabla\eta\}\|_{H^{-\frac{1}{2}}(\Gamma)}
+ \|\{\nabla\xi\}\|_{H^{-\frac{1}{2}}(\Gamma)}\|[\eta]\|_{H^{\frac{1}{2}}(\Gamma)}
\\[1ex] \nonumber
&\lesssim \|[\xi]\|_{H^{\frac{1}{2}}(\Gamma)}(\|\eta\|_{H^{1}([\frac{R}{2},R]\times S^2)}+\|\eta\|_{H^{1}(\Omega_{\rm out})})
+ (\|\xi\|_{H^{1}([\frac{R}{2},R]\times S^2)}+\|\xi\|_{H^{1}(\Omega_{\rm out})})\|[\eta]\|_{H^{\frac{1}{2}}(\Gamma)}
\\[1ex] \nonumber
&\lesssim (\|\xi\|_{H^{1}([\frac{R}{2},R]\times S^2)}+\|\xi\|_{H^{1}(\Omega_{\rm out})})
(\|\eta\|_{H^{1}([\frac{R}{2},R]\times S^2)}+\|\eta\|_{H^{1}(\Omega_{\rm out})})
\\[1ex]
&\lesssim \|\xi\|_{\rm DG}\|\eta\|_{\rm DG}.
\end{align}
Similarly, $I_3$ can be estimated by
\begin{eqnarray}\label{proof-r3}
I_3 \lesssim\sigma^{\frac{1}{2}}\|\xi\|_{\rm DG}\|[\eta]\|_{L^2(\Gamma)}.
\end{eqnarray}
	
Collecting \eqref{proof-r0} and the error bounds \eqref{proof-r1} to \eqref{proof-r3}, we have
\begin{eqnarray*}
\|\xi\|_{\rm DG} \lesssim \|\eta\|_{\rm DG}.
\end{eqnarray*}
We obtain from \eqref{rate-planewave} and \eqref{rate-sphere} that
if $u\in H^s(\Omega_{\rm out})\oplus H^s([0,R]\times S^2)~ (s\geq3)$, then
\begin{eqnarray*}
\|\eta\|_{H^{1}([\frac{R}{2},R]\times S^2)}+\|\eta\|_{H^{1}(\Omega_{\rm out})}
\lesssim \varrho^{1-s}\|w\|_{\hs},
\end{eqnarray*}
which together with
\begin{eqnarray*}
\|[\eta]\|_{L^2(\Gamma)}\lesssim \|\eta\|_{H^{\frac{1}{2}}([\frac{R}{2},R]\times S^2)}+\|\eta\|_{H^{\frac{1}{2}}(\Omega_{\rm out})}
\lesssim \varrho^{\frac{1}{2}-s}\|w\|_{\hs}
\end{eqnarray*}
leads to
\begin{eqnarray*}
\|w-w^{\rm DG}\|_{\rm DG} \leq \|\eta\|_{\rm DG}+\|\xi\|_{\rm DG} \lesssim
(\varrho^{1-s}+\varrho^{\frac{1}{2}-s}\gamma_\varrho)\|w\|_{\hs}.
\end{eqnarray*}
Then we can derive \eqref{rate-T-tildeT} by using \eqref{inverse-estimate}.
\end{proof}

\subsection{DG approximations of the eigenvalue problem}
\label{subsec-eigen}

We construct DG methods for eigenvalue problem \eqref{eq-model}: Find $\lambda^{\rm DG}
\in\mathbb{R}$ and $u^{\rm DG}\in \mathcal{S}^K_{NL}(\Omega)$, such that 
$\|u^{\rm DG}\|_{L^2_{\#}(\Omega)} =1$ and
\begin{eqnarray}\label{eq-eigen-DG}
a^{\rm DG}(u^{\rm DG},v)=\lambda^{\rm DG}(u^{\rm DG},v)
\qquad\forall~v\in\mathcal{S}^K_{NL}(\Omega).
\end{eqnarray}
Note that \eqref{eq-model} and \eqref{eq-eigen-DG} are equivalent to $\lambda Tu=u$
and $\lambda^{\rm DG}T^{\rm DG}u^{\rm DG}=u^{\rm DG}$, respectively.

Denote by $\sigma(T)$ the spectrum and $\rho(T)$ the resolvent set of the solution operator $T$.
For any $z\in\mathbb{C}$ in $\rho(T)$, we define the resolvent operator $R_z(T)=(z-T)^{-1}$.
Let $\lambda^{-1}$ be an eigenvalue of $T$ and $\gamma$ be a circle in the complex plane
that is centered at $\lambda^{-1}$ and does not enclose any other point of $\sigma(T)$. 

Define the following operators with contour integrations:
\begin{eqnarray*}
\mathscr{E}=\mathscr{E}(\lambda)=\frac{1}{2\pi i}\int_{\gamma}R_z(T)dz 
\quad{\rm and}\quad
\mathscr{E}^{\rm DG}=\mathscr{E}^{\rm DG}(\lambda)=\frac{1}{2\pi i}\int_{\gamma}R_z(T^{\rm DG})dz.
\end{eqnarray*}
If $\varrho$ is sufficiently large, then $\mathscr{E}$ and $\mathscr{E}^{\rm DG}$ are the spectral projectors
of $T$ and $T^{\rm DG}$ relative to $\lambda^{-1}$, respectively (see \cite{osborn75}).

Define the distances
\begin{eqnarray*}
\mathcal{D}(X,Y)=\sup_{x\in X\atop\|x\|_{\rm DG}=1}\inf_{y\in Y}\|x-y\|_{\rm DG}
\quad{\rm and}\quad
\mathscr{D}(X,Y)=\max\{\mathcal{D}(X,Y),\mathcal{D}(Y,X)\}.
\end{eqnarray*}
Using proposition \ref{proposition-T-approximate} and similar arguments as those in \cite{antonietti06}, we have the following
convergence results (including non-pollution and completeness) for DG eigenvalues and eigenspaces.
\begin{remark}\label{theo-convergence}
Let $A\subset\mathbb{R}$ be an open set containing $\sigma(T)$. If
$C_{\sigma}$ and $\varrho$ are sufficiently large, 
then $\sigma(T^{\rm DG})\subset A$.
Moreover, for all $z\in\sigma(T)$, we have
\begin{eqnarray*}
\lim_{\varrho\rightarrow\infty} \inf_{y\in\sigma(T^{\rm DG})}|z-y|=0.
\end{eqnarray*}
In addition, we have
\begin{eqnarray*}
\lim_{\varrho\rightarrow\infty}\mathscr{D}\big(\mathcal{R}(\mathscr{E}^{\rm DG}),\mathcal{R}(\mathscr{E})\big)=0,
\end{eqnarray*}
where $\mathcal{R}$ denotes the range.
\end{remark}

Now we can derive the following {\it a priori} error estimate for DG approximations.

\begin{theorem}\label{theo-approximation-rate}
Assume that \eqref{eq-inverse-H1-assumption} is true and  $C_{\sigma}$ is sufficiently large. 
Let $\lambda$ be an eigenvalue of \eqref{eq-model} with algebraic multiplicity m.
Then for $\varrho$ sufficiently large, there exist $m$ eigenpairs $(\lambda_{i}^{\rm DG},u_{i}^{\rm DG})~(i=1,2,\cdots m) $ of \eqref{eq-eigen-DG} such that
\begin{eqnarray}\label{eq-convergence-rate}
|\lambda_i^{\rm DG}-\lambda|+\|u_i^{\rm DG}-u_i\|_{\rm DG} \leq C_s\varrho^{\frac{3}{2}+\epsilon-s}
\qquad \forall~s\geq 3,~~i=1,2,\cdots m,
\end{eqnarray}
where the constant $C_s$ depends only on $\lambda, u_i$ and $s$.
\end{theorem}

\begin{proof}
Note that for $f\in L_{\#}^2(\Omega)$, $Tf\in \widetilde{H}^2(\Omega)$ and $\|Tf\|_{\widetilde{H}^2(\Omega)}\leq C\|f\|_{L_{\#}^2(\Omega)}$
(see \cite{egorov97}, p.257, Thm. 9 and (8.137)). 
	
For $f\in\hdel$, we have from Proposition \ref{proposition-T-approximate} that
\begin{eqnarray*}
\|(T-T^{\rm DG})f\|_{\rm DG}  
\lesssim \varrho^{-\frac{1}{2}+\epsilon} \|Tf\|_{\widetilde{H}^2(\Omega)}  
\lesssim \varrho^{-\frac{1}{2}+\epsilon}\|f\|_{L_{\#}^2(\Omega)}  
\lesssim 
\varrho^{-\frac{1}{2}+\epsilon}\|f\|_{\rm DG},
\end{eqnarray*}
which implies
\begin{eqnarray}\label{proof-8-1}
\lim_{\varrho\rightarrow \infty}\|T-T^{\rm DG}\|_{\mathscr{L}(\hdel,\hdel)}
\leq C\lim_{\varrho\rightarrow \infty}\varrho^{-\frac{1}{2}+\epsilon} = 0.
\end{eqnarray}
Using \eqref{proof-8-1} and \cite[Theorem 1]{osborn75}, we have the convergence of the eigenvalues and
\begin{eqnarray}\label{distance_T}
\mathscr{D}(\mathcal{R}(\mathcal{E}),\mathcal{R}(\mathcal{E}^{\rm DG}))
\lesssim \|T-T^{\rm DG}\|_{\mathscr{L}(\mathcal{R}(\mathcal{E}),{\hdel})}.
\end{eqnarray}
Then it is only necessary for us to estimate the right-hand side of \eqref{distance_T}.

Using proposition \ref{proposition-T-approximate}, the regularity result Lemma \ref{lemma-regularity}
and the fact $Tv=\lambda^{-1}v$ for $v\in \mathcal{R}(\mathcal{E})$, we have that for any $s\geq 3$,
\begin{multline}\label{add-proof-21}
\|T-T^{\rm DG}\|_{\mathscr{L}(\mathcal{R}(\mathscr{E}),\hdel)}
= \sup_{v\in \mathcal{R}(\mathscr{E}),\|v\|_{\rm DG}=1}\|(T-T^{\rm DG})v\|_{\rm DG} \\[1ex]
\leq C\varrho^{\frac{3}{2}+\epsilon-s}\sup_{v\in \mathcal{R}(\mathscr{E}),\|v\|_{\rm DG}=1}
\|Tv\|_{\hs}
\leq C\varrho^{\frac{3}{2}+\epsilon-s}\sup_{v\in \mathcal{R}(\mathscr{E}),\|v\|_{\rm DG}=1}
\|v\|_{\hs} 
\leq C_s\varrho^{\frac{3}{2}+\epsilon-s} ,
\end{multline}
where $C_s$ is a constant depending only on $\mathcal{R}(\mathscr{E}), \lambda$  and $s$ .

It is apparent from \eqref{distance_T} and \eqref{add-proof-21} that
\begin{eqnarray}\label{distance_0}
\lim\limits_{\varrho\rightarrow \infty}\mathscr{D}(\mathcal{R}(\mathcal{E}),\mathcal{R}(\mathcal{E}^{\rm DG}))
=0.
\end{eqnarray}
Let $m$ and $m_{\varrho}$ be the dimensions of $\mathcal{R}(\mathcal{E})$ and $\mathcal{R}(\mathcal{E}^{\rm DG})$, respectively. 
Then, \eqref{distance_0} indicates that, for $\varrho$ large enough, $m=m_{\varrho}$ (see \cite[p.200]{Kato66}) and there exist $m$ eigenfunctions $u_{i}^{\rm DG}\in \mathcal R(\mathcal{E}^{\rm DG})$ 
and $m$ eigenpairs $(\lambda_{i}^{\rm DG},u_{i}^{\rm DG})~(i=1,2,\cdots m) $ satisfying \eqref{eq-eigen-DG}. 
Moreover, according to the definition of distance $\mathscr{D}(X,Y)$, 
we can find $u_i\in \mathcal R(\mathcal{E})$  and $\|u_i\|_{L^2_{\#}(\Omega)}=1 $ such that 
\begin{eqnarray}
\|u_i^{\rm DG}-u_i\|_{\rm DG}\lesssim\mathscr{D}(\mathcal{R}(\mathcal{E}),\mathcal{R}(\mathcal{E}^{\rm DG}))\qquad i=1,2,\cdots m.
\end{eqnarray} 
This completes the proof of error estimates for eigenfunctions.

For eigenvalues, we obtain by a simple calculation that
\begin{align}\label{eigen-identity} 
\nonumber
&
\lambda-\lambda_i^{\rm DG} = a(u_i,u_i)-a^{\rm DG}(u_i^{\rm DG},u_i^{\rm DG}) 
\\[1ex]\nonumber
&= a^{\rm DG}(u_i-u_i^{\rm DG},u_i-u_i^{\rm DG})+2a^{\rm DG}(u_i^{\rm DG},u_i-u_i^{\rm DG})+a^{\rm DG}(u_i,u_i^{\rm DG})-a^{\rm DG}(u_i^{\rm DG},u_i) 
\\[1ex]\nonumber
&= a^{\rm DG}(u_i-u_i^{\rm DG},u_i-u_i^{\rm DG})+2\lambda_i^{\rm DG}(u_i^{\rm DG},u_i-u_i^{\rm DG})+2D_{\delta}+a^{\rm DG}(u_i,u_i^{\rm DG})-a^{\rm DG}(u_i^{\rm DG},u_i) 
\\[1ex]
&= a^{\rm DG}(u_i-u_i^{\rm DG},u_i-u_i^{\rm DG})-\lambda_i^{\rm DG}(u_i-u_i^{\rm DG},u_i-u_i^{\rm DG})+D_{\delta}+\overline{D_{\delta}}
\end{align}
with the consistency error
\begin{eqnarray}\label{add-D-error} \nonumber
D_{\delta} &=& a^{\rm DG}(u_i^{\rm DG},u_i)-\lambda_i^{\rm DG}(u_i^{\rm DG},u_i) =
a^{\rm DG}(u_i^{\rm DG},u_i-u_{i}^{\rm DG})-\lambda_i^{\rm DG}(u_i^{\rm DG},u_i-u_{i}^{\rm DG}) \\[1ex]
&\leq& C\varrho^{\frac{3}{2}+\epsilon-s}(\|u_i^{\rm DG}\|_{\rm DG}+\lambda_i^{\rm DG})\|u_i\|_{\hs}.
\end{eqnarray}
Using \eqref{add-proof-21} to \eqref{add-D-error},
we obtain
\begin{eqnarray*}
|\lambda-\lambda_i^{\rm DG}|\leq C_{s} \varrho^{\frac{3}{2}+\epsilon-s} \quad\quad\forall~s\geq 3,
\end{eqnarray*}
where the constant $C_s$ depends only on $\lambda, u_i$ and $s$.
\end{proof}

\begin{remark}
	We emphasize that our result works not only for the case of single eigenvalue ($m=1$),	but also for general cases of multiple eigenvalue ($m>1$).
\end{remark}

\begin{remark}
It is shown in many cases that the convergence rate of finite dimensional approximations under a weaker norm is faster than that under a stronger norm (see, e.g., \cite{babuska91,chatelin83}).
By making this assumption for our DG approximations, for example, 
\begin{eqnarray*}
	\|u_i-u_{i}^{\rm DG}\|_{L^2(\Omega)} \lesssim \varrho^{-\alpha} \|u_i-u_{i}^{\rm DG}\|_{\rm DG} \qquad\text{with some } \alpha>0,
\end{eqnarray*}
it may be true from  \eqref{eigen-identity} and \eqref{add-D-error} that the eigenvalue approximations have better convergence rate than that of eigenfunctions.

\end{remark}

\begin{remark}	
	Within the framework of Kohn-Sham density functional theory, one has to solve the nonlinear eigenvalue problem \eqref{eq:eigen}  with a SCF iteration. 
	Using our DG discretizations, the linear eigenvalue problem \eqref{eq-eigen-DG} is solved at each iteration step
	and complex mixing schemes such as Roothaan,  level-shifting and DIIS algorithms (see, e.g., \cite{cances00,lebris03}) are used to achieve convergence.
	
	If the exchange-correlation potential $V_{\rm xc}$ is sufficiently smooth and the trial state (from previous DG approximations)
	$\tilde{\Phi}\in(\mathcal{S}^K_{NL}(\Omega))^{\Ne}$, then we have from similar arguments as those in \cite{flad08} that 
	the eigenfunctions $\{\phi_i\}_{i=1,\cdots,\Ne}$ of $H_{\tilde{\Phi}}$ belong to 
	$C_{\rm loc}^{\infty}(\mathbb{R}^3\setminus\mathcal{R}) \cup C^{\infty}([0,R]\times S^2)$.
	This regularity together with the analysis in Theorem \ref{theo-approximation-rate} gives spectral convergence rates for 
	DG approximations of the (linear) eigenvalue problem in each SCF iteration step.
	
	Note that we have not obtained {\it a priori} error estimates for approximations of nonlinear eigenvalue problems 
	but only for linearized equations in SCF iterations.
	We refer to \cite{cances12,chen13} for numerical analysis of nonlinear eigenvalue problems.
\end{remark}

\section{Numerical experiments}
\label{sec-numerical} 
\setcounter{equation}{0}
\setcounter{figure}{0}

In this section, we will present some details for implementing our DG scheme,
and some numerical experiments in electronic structure calculations.

\subsection {Hamiltonian matrix elements}
\label{subsec-integral}

With our DG scheme, we can discretize the continuous eigenvalue problem into a (finite dimensional) matrix generalized eigenvalue problem
\begin{eqnarray*}
H\hat{u}_i = \lambda_i M \hat{u}_i ,
\end{eqnarray*}
where $\hat{u}_i$ are eigenvectors that correspond to the DG approximations $u_i^{\rm DG}$.
We shall explain in the following how the matrix elements of $H$ and $M$ are generated.

For basis functions ${\bf p}$ and ${\bf q}$, We divide the integrals for overlap matrix $M_{\bf pq}$
and stiff matrix $H_{\bf pq}$ into three parts.
The scattering identity (see, e.g., \cite{messiah64})
\begin{eqnarray}\label{eq-scattering}
e^{i\bf{k}\cdot\bf{r}}=4\pi\sum_{lm}i^l j_l(kr)\tilde{Y}_{lm}^*({\bf k})\tilde{Y}_{lm}({\bf r})
\end{eqnarray}
with $k=|{\bf k}|$,
is heavily used to bridge the gap between plane waves and spherical harmonics.

For ${\bf p,q}\in\mathcal{P}_K(\Omega_{\rm out})$, we have
\begin{eqnarray}\label{eq-matrix-Ma}
M^a_{\bf pq} = (e_{\bf {k_q}}|_{\Omega_{\rm out}},e_{\bf {k_p}}^*|_{\Omega_{\rm out}})
=\frac{1}{|\Omega|}\int_{\Omega_{\rm out}} e^{i({\bf k_q}-{\bf k_p})\cdot {\bf r}}
=U({\bf k_q}-{\bf k_p}), 
\end{eqnarray}
where $U({\bf k})$ is the Fourier transform of the step function with 0 inside the sphere and 1 outside
\begin{eqnarray*}\label{eq-U}
	U({\bf k})=\frac{1}{|\Omega|}\int_{\Omega_{\rm out}} e^{i{\bf k}\cdot {\bf r}}=\left\{ 
	\begin{array}{ll} |\Omega_{\rm out}|/|\Omega| &\mbox{if } k=0, \vspace{0.3cm}\\[1ex]
		-4\pi R^2 j_1(kR)/(k|\Omega|) &\mbox{if } k\neq 0 
	\end{array} \right.
\end{eqnarray*}
with $k=|{\bf k}|$ and $j_l$ the $l$th spherical Bessel function.
Similarly, we have from \eqref{eq-bilinear-DG} that
\begin{eqnarray}\label{eq-matrix-Ha}
H^a_{\bf pq} 
&=& a^{\rm DG}(e_{\bf k_q}|_{\Omega_{\rm out}},e^*_{\bf k_p}|_{\Omega_{\rm out}})
= \frac{1}{2}{\bf k_p\cdot\bf k_q}U({\bf k_q}-{\bf k_p})+V({\bf k_q}-{\bf k_p}) + \mathfrak{D}^a_{\bf pq},
\end{eqnarray}
where
\begin{eqnarray}\label{eq-V}
V({\bf k})=\frac{1}{|\Omega|^{\frac{1}{2}}}\int_{\Omega} V_{\rm eff}({\bf r})e_{\bf k}
-\frac{4\pi}{|\Omega|}\sum_{lm} i^l\tilde{Y}_{lm}({\bf k})\int_0^R r^2 v_{lm}(r) j_{l}(kr)dr 
\end{eqnarray}
with $k=|{\bf k}|$ and the potential inside the sphere expanded by $V({\bf r})=\sum_{lm}v_{lm}(r)\tilde{Y}_{lm}({\bf r})$.
The discontinuity and penalization term $\mathfrak{D}^a_{\bf pq}$ in \eqref{eq-matrix-Ha} is given by
\begin{align}\label{dg-pw}
\mathfrak{D}^a_{\bf pq}=\left\{ 
\begin{array}{ll}
\displaystyle\frac{4\pi R^2\sigma}{|\Omega|} & k_{\bf p}=k_{\bf q}=0, \vspace{0.3cm}\\[1ex]
\displaystyle\frac{4\pi R^2}{|\Omega|}\left(\frac{1}{4}\frac{\partial j_0(k_{\bf q}r)}{\partial r}\bigg|_{r=R}+ \sigma j_0(k_{\bf q}R)\right) & k_{\bf p}=0,~k_{\bf q}\neq 0,\vspace{0.3cm} \\[1ex]
\displaystyle\frac{4\pi R^2}{|\Omega|}\left(\frac{1}{4}\frac{\partial j_0(k_{\bf p}r)}{\partial r}\bigg|_{r=R} + \sigma j_0(k_{\bf p}R)\right) & k_{\bf p}\neq 0,~k_{\bf q}=0, \vspace{0.3cm}\\[1ex]
\displaystyle\frac{(4\pi R)^2}{|\Omega|}\sum_{lm}(-1)^l\tilde{Y}^*_{lm}(-{\bf k}_{\bf p})\tilde{Y}_{lm}({\bf k}_{\bf q}) \left(\frac{1}{4}j_l(k_{\bf q}R)\frac{\partial j_l(k_{\bf p}r)}{\partial r}\bigg|_{r=R}\right.\vspace{0.3cm} & \\[1ex]
\qquad\left.\displaystyle +\frac{1}{4}j_l(k_{\bf p}R)\frac{\partial j_l(k_{\bf q}r)}{\partial r}\bigg|_{r=R} + \sigma j_l(k_{\bf p}R)j_l(k_{\bf q}R)\right)
& k_{\bf p}\neq 0,~k_{\bf q}\neq 0, \\[1ex]
\end{array}\right.
\end{align}
with $k_{\bf p}=|{\bf k}_{\bf p}|$ and $k_{\bf q}=|{\bf k}_{\bf q}|$.
Note that the first term of \eqref{eq-V} is obtained by fast Fourier transform (FFT) and the second term is calculated by numerical integrations.

For ${\bf p,q}\in\mathcal{B}_{NL}(\Omega_{\rm in})$, we have from the orthogonality of $Y_{lm}$ on the surface that
\begin{eqnarray}\label{eq-matrix-Mb}
M^b_{\bf pq}=\delta_{ll'}\delta{mm'} \int_0^R r^2\chi_n(r)\chi_{n'}(r)dr
\end{eqnarray}
and
\begin{eqnarray}\label{eq-matrix-Hb}\nonumber
H^b_{\bf pq} &=& a^{\rm DG}\big( \chi_{n'}(r)\tilde{Y}_{l'm'}({\bf r})|_{\Omega_{\rm in}},
\chi_{n}(r)\tilde{Y}^*_{lm}({\bf r})|_{\Omega_{\rm in}} \big) 
\\\nonumber 
&=& \delta_{ll'}\delta{mm'}
\int_0^R \frac{1}{2}\left(r^2\chi'_n(r)\chi'_{n'}(r) + l(l+1)\chi_n(r)\chi_{n'}(r)\right)dr
\\
&& + \sum_{\hat{l}\hat{m}} G(ll'\hat{l},mm'\hat{m}) 
 \int_0^R r^2\chi_n(r)\chi_{n'}(r)v_{\hat{l}\hat{m}}(r)dr
+ \mathfrak{D}^b_{\bf pq},
\end{eqnarray}
where the potential inside the sphere is expanded by
$V({\bf r})=\sum_{\hat{l}\hat{m}}v_{\hat{l}\hat{m}}(r)\tilde{Y}_{\hat{l}\hat{m}}({\bf r})$
and $G$ is the integral of three spherical harmonics that can be written in terms of Gaunt coefficients (see, e.g., \cite{martin05}).
The discontinuity and penalization term $\mathfrak{D}^b_{\bf pq}$ is 
\begin{eqnarray}\label{dg-r}
\mathfrak{D}^b_{\bf pq}=\delta_{ll'}\delta{mm'}R^2\left(-\frac{1}{4}\chi_n(R)\chi'_{n'}(R)-
\frac{1}{4}\chi'_n(R)\chi_{n'}(R)+\sigma\chi_n(R)\chi_{n'}(R)\right).
\end{eqnarray}

For ${\bf p}\in\mathcal{P}_K(\Omega_{\rm out}),~{\bf q}\in\mathcal{B}_{NL}(\Omega_{\rm in})$, we have
\begin{eqnarray}\label{eq-matrix-Mc}
M^c_{\bf pq}=0
\end{eqnarray}
and
\begin{align}\label{eq-matrix-Hc} \nonumber
& H^c_{\bf pq} ~=~ a^{\rm DG}(\chi_{n}(r)\tilde{Y}_{lm}({\bf r})|_{\Omega_{\rm in}},e^*_{\bf k_p}|_{\Omega_{\rm out}}) ~=~ \mathfrak{D}^c_{\bf pq}
\\[1ex] 
&=\left\{
\begin{array}{ll}
0  & k_{\bf p}=0,l\neq0, \vspace{0.3cm}\\[1ex]
\displaystyle \frac{\sqrt{4\pi}R^2}{|\Omega|^{\frac{1}{2}}} \left(\frac{1}{4}\chi'_{n}(R)
-\sigma\chi_{n}(R)\right)  & k_{\bf p}=0,l=0, \vspace{0.3cm} \\[1ex]
\displaystyle
\frac{4\pi R^2}{|\Omega|^{\frac{1}{2}}}~i^l~ \tilde{Y}_{lm}(-{\bf k}_{\bf p})\left(\frac{1}{4}j_l(k_{\bf p}R)\chi'_{n}(R)
-\frac{1}{4}\chi_{n}(R)\frac{\partial j_l(k_{\bf p}r)}{\partial r}\bigg|_{r=R} \right. &~  \vspace{0.3cm}
\\[1ex]
\qquad -\sigma j_l(k_{\bf p}R)\chi_{n}(R)\bigg) & k_{\bf p}\neq 0.
\end{array}\right. \quad\quad
\end{align}

Since we use a symmetric DG scheme, the elements for 
${\bf p}\in\mathcal{B}_{NL}(\Omega_{\rm in}),~{\bf q}\in\mathcal{P}_K(\Omega_{\rm out})$ 
can be obtained immediately.

Combining \eqref{eq-matrix-Ma} -- \eqref{eq-matrix-Hc}, we can obtain the matrices $H$ and $M$,
and further solve the matrix eigenvalue problems by linear eigensolvers.

\subsection{Numerical results}
\label{subsec-numerial-results}

All the numerical results are presented by atomic units (a.u.).
When we test the convergence with respect to one parameter (say, $K$, $N$ or $L$), the other two parameters are fixed and chosen to be sufficiently large.

\noindent{\bf Example 1.} (linear problem for a single-atom system)
Consider the linear eigenvalue problem: Find $\lambda\in\mathbb{R}$ and $u\in H^1_{\#}(\Omega)$ such that
\begin{eqnarray}\label{toy_model_1}
\left(-\frac{1}{2}\Delta +V \right)u = \lambda u,
\end{eqnarray}
with $\Omega=[-5,5]^3$ and 
$\displaystyle V({\bf r})=-\frac{4\pi}{|\Omega|}\sum_{{\bf k}\in\mathcal{R}^*,{\bf k}\neq 0} \frac{e^{i{\bf k}\cdot {\bf r}}}{|{\bf k}|^2}$ .
Note that the potential $V$ can be viewed as a periodized version of the potential $\displaystyle -\frac{1}{|\bf r|}$ for a hydrogen atom.
It is periodic and sufficiently smooth everywhere except at the origin.
Then due to Lemma \ref{lemma-regularity},
the error estimates in Theorem \ref{theo-approximation-rate} hold.

We first compare the numerical errors of the lowest eigenvalue approximations by plane waves 
and our DG methods (see Figure \ref{fig-pw-dg}),
from which we observe that the DG approximations converge much faster.
We compare the eigenfunctions along the $x$-axis obtained by plane waves and DG methods (see Figure \ref{fig-wavefunction}).
We observe that the DG approximations can capture the cusp at the nuclear position while that plane waves can not.
For a more precise comparison, when the required accuracy is $10^{-1}$ (for the first eigenvalue), the DG method needs around $50$ degrees of freedom (DOFs) while the plane waves method need about $60$ DOFs; when the required accuracy is $10^{-2}$, the DG method needs around 300 DOFs while the plane waves method needs more than 1100 DOFs.

We further show the convergence rates of the eigenvalue errors with respect to plane wave truncations $K$
(see Figure \ref{fig-pw-dg-2}),  and observe exponential decay for different sizes of atomic spheres. 
We find a slightly faster convergence rate of the numerical errors (in Figure \ref{fig-pw-dg-2}) with a bigger size of atomic sphere.
The reason is that the eigenfunctions are less varying outside a larger atomic sphere.
However, we see that the choice of $R$ does not affect the numerical simulations significantly. In practical simulations, we could choose relatively large atomic spheres as long as they do not overlap.

We also present the numerical errors with respect to the orders of radial basis functions (see Figure \ref{fig-r}),
and compare the polynomials with Slater-type atomic orbitals
\begin{eqnarray*}
 \chi_k(r)=r^k e^{-\eta r}, \quad k=0,1,\cdots
\end{eqnarray*}
with $\eta$ a fixed parameter.
We observe that although a high accuracy can be obtained by Slater-type atomic orbitals with very few degrees of freedom, 
a better systematically convergence rate is achieved by polynomials.

\begin{figure}[htb!]
\begin{minipage}[t]{0.5\linewidth}
\centering
\includegraphics[width=6.5cm]{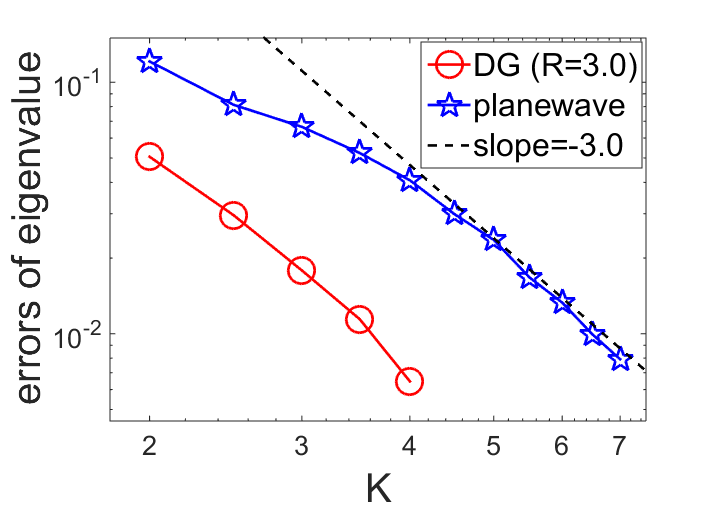}
\caption{(Example 1) Numerical errors of plane waves and DG approximations in the single-atom system.}
\label{fig-pw-dg}
\end{minipage}
\hskip 0.5cm
\begin{minipage}[t]{0.5\linewidth}
	\centering
	\includegraphics[width=6.5cm]{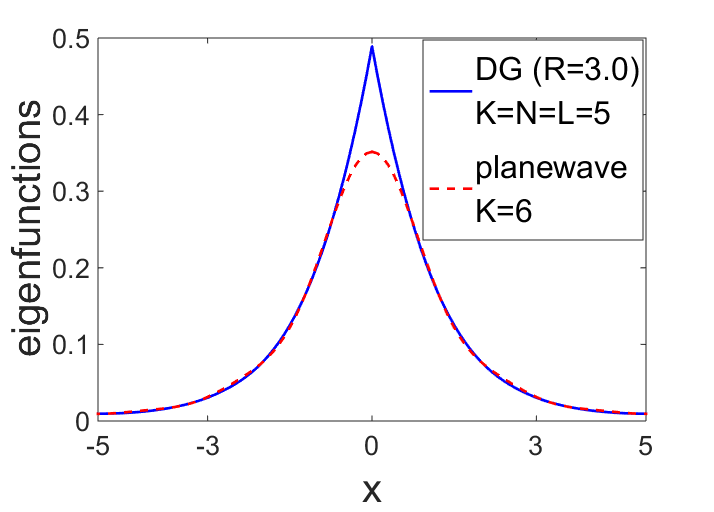}
	\caption{(Example 1) Eigenfunctions along the $x$-axis obtained by plane waves and DG discretizations.}
	\label{fig-wavefunction}
\end{minipage}
\vskip 0.2cm
	\begin{minipage}[t]{0.5\linewidth}
		\centering
		\includegraphics[width=7.0cm]{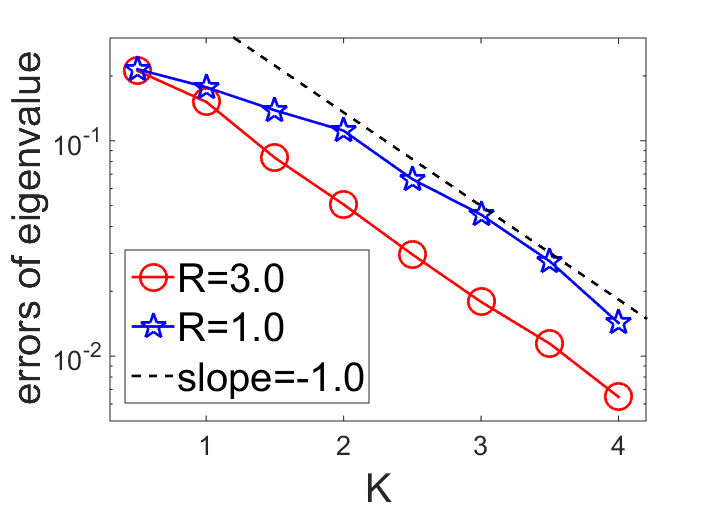}
		\caption{(Example 1) Numerical errors of DG approximations with respect to $K$.}
		\label{fig-pw-dg-2}
	\end{minipage}
\hskip 0.5cm
\begin{minipage}[t]{0.5\linewidth}
\centering
\includegraphics[width=7.0cm]{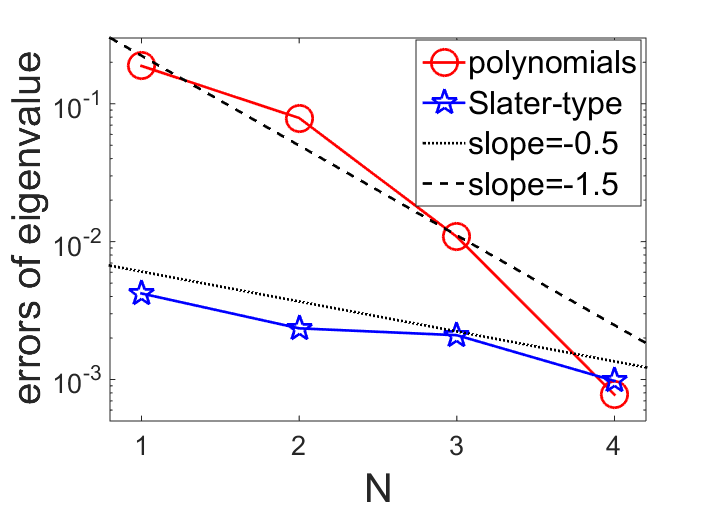}
\caption{(Example 1) Numerical errors for different types of radial basis functions.}
\label{fig-r}
\end{minipage}
\end{figure}

\noindent{\bf Example 2.} (linear problem for a two-atom system)
Consider the linear eigenvalue problem for a two-atom system: Find $\lambda\in\mathbb{R}$ and $u\in H^1_{\#}(\Omega)$ such that
\begin{eqnarray}\label{toy_model_2}
\left(-\frac{1}{2}\Delta +V_{1}+V_{2} \right)u = \lambda u,
\end{eqnarray}
where  $\Omega=[-5,5]^3$ and
$\displaystyle V_{j}({\bf r}) = -\frac{4\pi}{|\Omega|}\sum_{{\bf k}\in\mathcal{R}^*,{\bf k}\neq 0} \frac{1}{|{\bf k}|^2} e^{i{\bf k}\cdot ({\bf r-R_j})}~ (j=1,2) $
with ${\bf R_1}$ and ${\bf R_2}$ the positions of atoms.

We first compare the numerical errors of the the lowest eigenvalue approximations  by plane waves and our DG methods.
We observe a much better convergence rate with respect to $K$ in Figure \ref{fig-2atom-err} and a more accurate capture of the eigenfunction cusp by our DG approximation in Figure \ref{fig-2atom-eigenfunction}.
For a more precise comparison, when the required accuracy is $10^{-1}$, the DG method needs around 100 DOFs while the plane waves method need about 30 DOFs; when the required accuracy is $10^{-2}$, the DG method needs around 400 DOFs while the plane waves method needs more than 1000 DOFs.

We then show the convergence rates of the eigenvalue errors with respect to plane wave truncations $K$ (see Figure \ref{fig-pw-dg-2}),  and angular momentum truncation $L$ (see Figure \ref{fig-2atom-L})
We observe exponential decay with respect to both $K$ and $L$ for different sizes of atomic spheres. 
In this example, when the radii of atomic spheres are larger than 0.5, 
increasing the size of spheres does not improve the convergence rate significantly (see Figure \ref{fig-2atom-err_k}).
From the comparisons of different radii, we get the same conclusion as Example 1 that the choice of $R$ is not important in our scheme, and one can simply choose relative large radii such that the atomic spheres do not overlap.

We further show the convergence rates of numerical errors for the lowest 3 eigenvalues,
and observe exponential decay of the numerical errors with respect to $K$ (see Figure \ref{fig-2atom-eigenvalues}).
This supports our theory.

Finally, we test the effect of penalty parameter $C_{\sigma}$. 
In our DG scheme, the penalty constant $C_{\sigma}$ plays an important role to guarantee stability.
The errors with respect to different choices $C_{\sigma}$ are shown in Figure \ref{fig-err-sigma}. We observe that the DG method can be stable and accurate in a large range of values beyond a certain threshold value. 
Similar discussions for the penalty parameter can also be found in \cite{lin12}.

\begin{figure}[htb!]
	\begin{minipage}[t]{0.5\linewidth}
		\centering
		\includegraphics[width=6.5cm]{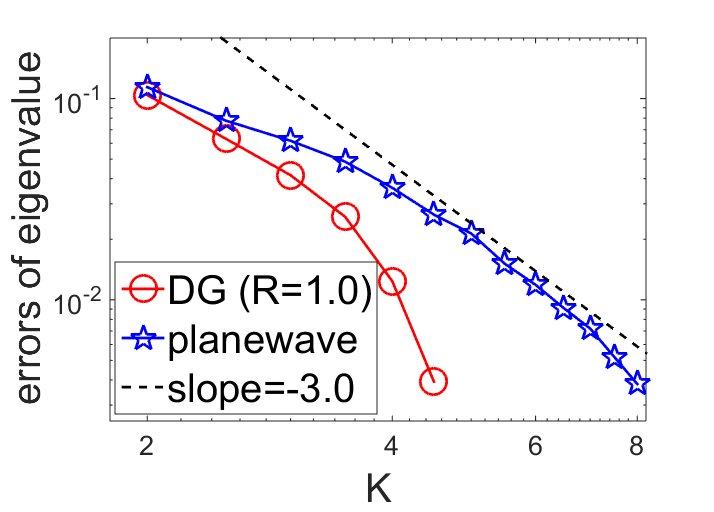}
		\caption{(Example 2) Numerical errors of plane waves and DG approximations in the two-atom system.}
		\label{fig-2atom-err}
	\end{minipage}
	\hskip 0.5cm
	\begin{minipage}[t]{0.5\linewidth}
		\centering
		\includegraphics[width=6.5cm]{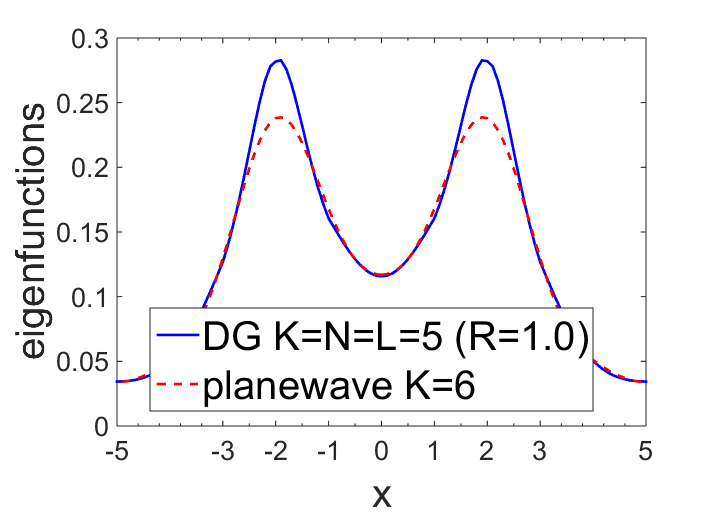}
		\caption{(Example 2) Eigenfunctions along the $x$-axis obtained by plane waves and DG discretizations.}
		\label{fig-2atom-eigenfunction}
	\end{minipage}
\vskip 0.2cm
	\begin{minipage}[t]{0.5\linewidth}
		\centering
		\includegraphics[width=6.5cm]{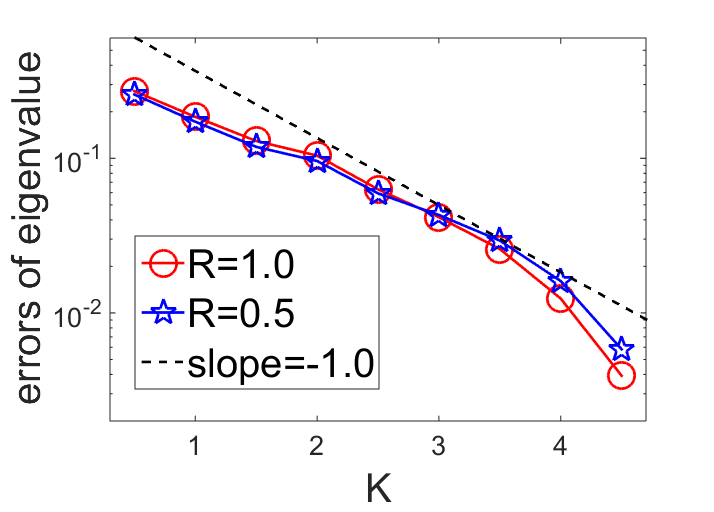}
		\caption{(Example 2) Numerical errors with respect to $K$ for different $R$ in the two-atom system.}
		\label{fig-2atom-err_k}
	\end{minipage}
	\hskip 0.5cm
	\begin{minipage}[t]{0.5\linewidth}
		\centering
		\includegraphics[width=6.5cm]{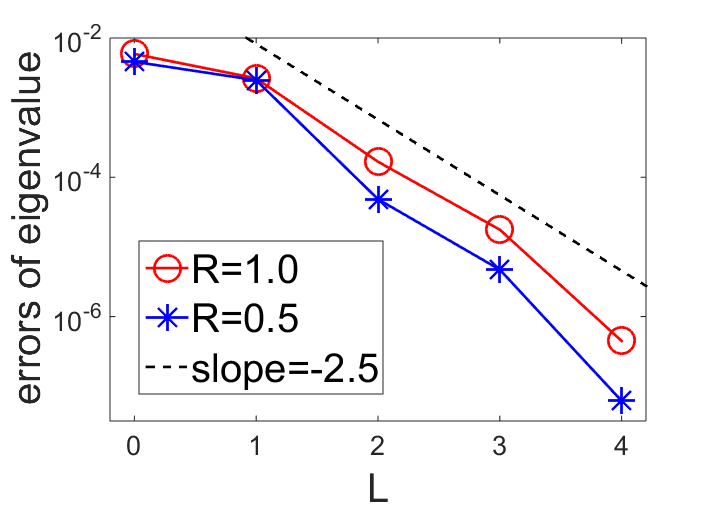}
		\caption{(Example 2) Numerical errors with respect to $L$ for different $R$ in the two-atom system.}
		\label{fig-2atom-L}
	\end{minipage}		
\vskip 0.2cm
	\begin{minipage}[t]{0.5\linewidth}
		\centering
		\includegraphics[width=6.5cm]{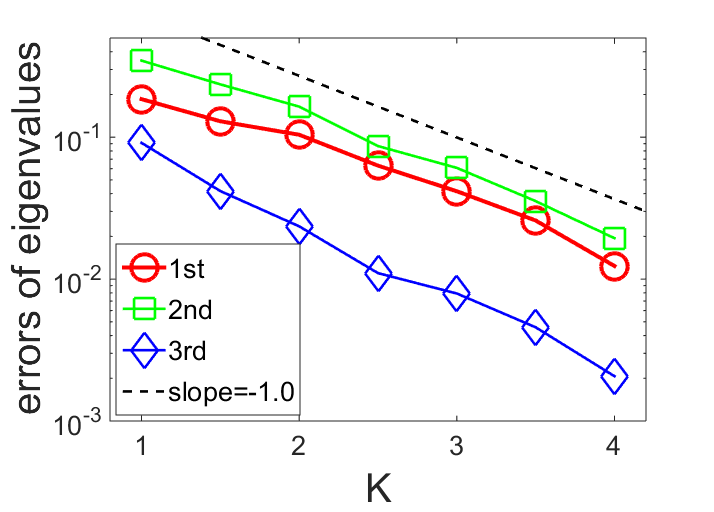}
		\caption{(Example 2) Numerical errors with respect to K for different eigenvalues in the two-atom system.}
		\label{fig-2atom-eigenvalues}
	\end{minipage}		
\hskip 0.5cm
\begin{minipage}[t]{0.5\linewidth}
\centering
\includegraphics[width=6.5cm]{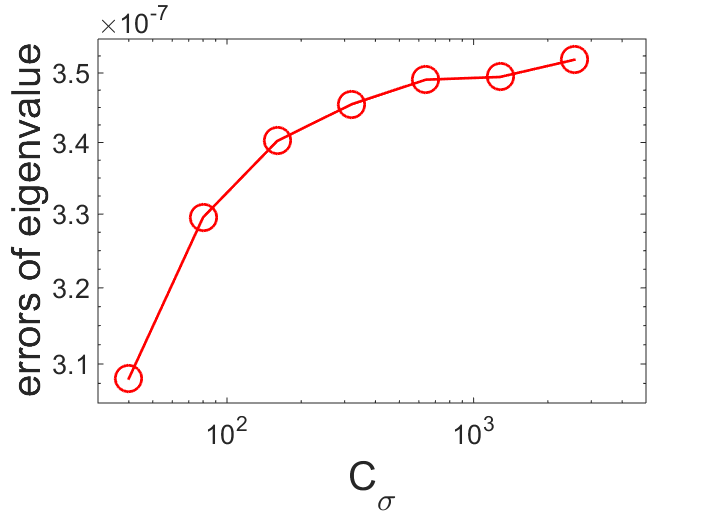}
\caption{(Example 2) Numerical errors with respect to the penalty parameter $C_{\sigma}$.}
\label{fig-err-sigma}
\end{minipage}		
\end{figure}

\noindent{\bf Example 3.} (simulation of a helium atom)
Consider the following nonlinear eigenvalue problem: Find $\lambda\in\mathbb{R}$ and $u\in H^1_{\#}(\Omega)$ such that
\begin{eqnarray}\label{toy_model_3}
\left(-\frac{1}{2}\Delta +V_{\rm ext}({\bf r})+V_{\rm H} [\rho] \right)u = \lambda u,
\end{eqnarray}
with the external potential  $\displaystyle V_{\rm ext}({\bf r})=-\frac{8\pi}{|\Omega|}\sum_{{\bf k}\in\mathcal{R}^*,{\bf k}\neq 0} \frac{e^{i{\bf k}\cdot {\bf r}}}{|{\bf k}|^2}$, the Hartree potential 
$\displaystyle V_{\rm H} [\rho]=\int_{\mathbb{R}^3}\frac{\rho({\bf y})}{|\cdot-{\bf y}|}d{\bf y}$ and $\rho=2u^2$.
Note that the exchange-correlation potential $V_{\rm xc}[\rho]$ is ignored here from a standard Kohn-Sham DFT model.
We present the eigenfunction along the $x$-axis (see Figure \ref{fig-He-eigenfunction}) 
and the convergence rates of numerical errors for different sizes of atomic spheres (see Figure \ref{fig-He-err}). 
We observe exponential decay of the numerical errors with respect to $K$ even for the nonlinear problem.

\begin{figure}[htb!]
	\begin{minipage}[t]{0.5\linewidth}
		\centering
		\includegraphics[width=6.5cm]{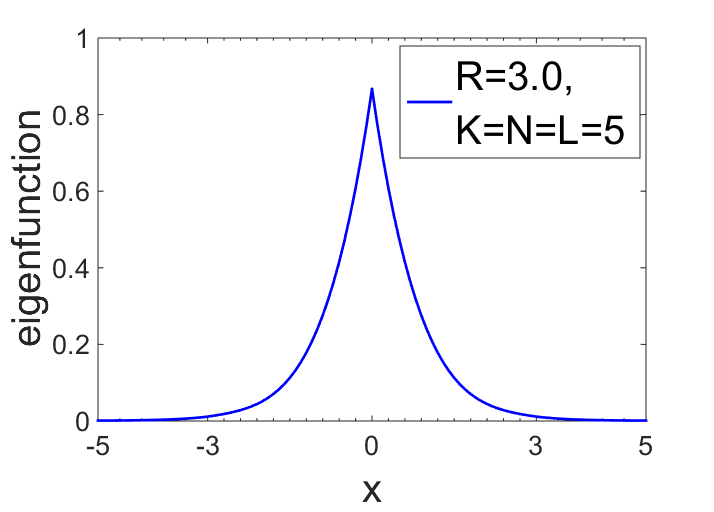}
		\caption{(Example 3) Eigenfunction along the $x$-axis obtained by DG discretizations in the helium-atom system.}
		\label{fig-He-eigenfunction}
	\end{minipage}		
	\hskip 0.5cm
	\begin{minipage}[t]{0.5\linewidth}
		\centering
		\includegraphics[width=6.5cm]{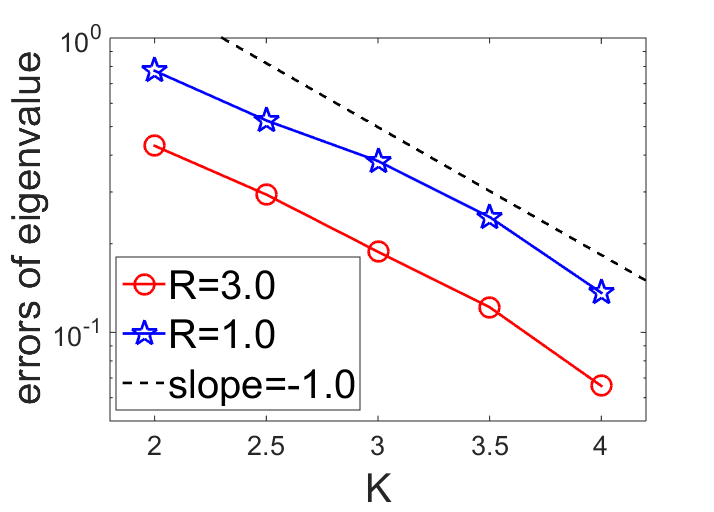}
		\caption{(Example 3) Numerical errors with respect to K for different R in the helium-atom system.}
		\label{fig-He-err}
	\end{minipage}	
\end{figure}

\section{Concluding remarks}
\label{sec-conclusions}

In this paper, we construct a discontinuous Galerkin scheme for full-potential electronic structure calculations.
It exploits the idea of augmented plane wave method which approximates the wavefunction in some ways ``the best of two worlds".
The smoothly varying parts of the wavefunctions between the atoms are represented by plane waves,
the rapidly varying parts near the nuclei are represented by radial atomic functions times spherical harmonics
inside a sphere around each nucleus, and these two parts are patched together by discontinuous Galerkin scheme.
We demonstrate {\it a priori} error estimate of this approximation to illustrate the accuracy and efficiency of this scheme,
and provide some numerical experiments to support the theory.

Besides the accuracy and efficiency we have shown in this paper,
the discontinuous Galerkin scheme is also flexible and economical for adaptive procedures
since the nonconformity results assuredly in limiting the contamination only to the subdomain where refinement is needed.
The {\it a posteriori} error analysis and the adaptive algorithm will be addressed in our future works.

\section*{Acknowledgement}
We are grateful to Reinhold Schneider in Technische Universit\"{a}t Berlin 
for his help in this work and for inspiring discussions.

\appendix
\renewcommand\thesection{\appendixname~\Alph{section}}

\section{Inverse estimates on the surface}
\label{sec:inverse-estimate}
\renewcommand{\theequation}{A.\arabic{equation}}
\renewcommand{\thetheorem}{A.\arabic{theorem}}
\renewcommand{\thelemma}{A.\arabic{lemma}}
\renewcommand{\theproposition}{A.\arabic{proposition}}
\renewcommand{\thealgorithm}{A.\arabic{algorithm}}
\renewcommand{\theremark}{A.\arabic{remark}}
\renewcommand{\thefigure}{A.\arabic{figure}}
\setcounter{equation}{0}
\setcounter{figure}{0}

In this appendix, we shall provide the numerical tests to support the inverse estimate assumption \eqref{eq-inverse-H1-assumption} and the ``interpolation" arguments to obtain \eqref{inverse-estimate}.

Let
\begin{eqnarray*}
\widetilde{V}_{\varrho}(\Gamma):=
\left\{ v_{\varrho} ~\big|~ \exists v\in\mathcal{S}^K_{NL}(\Omega) 
~{\rm such~that}~ v_{\varrho} = v^{+}+v^{-} \right\}.
\end{eqnarray*}
Consider the largest eigenvalue $\lambda_{\varrho,{\rm max}}$ of 
the following discrete eigenvalue problem on the spherical surface $\Gamma$:
Find $\lambda_{\varrho}\in\mathbb{R}$ and $u_{\varrho}\in\widetilde{V}_{\varrho}(\Gamma)$ such that
\begin{eqnarray}\label{eig-dis-gamma}
(-\Delta_{S^2}u_{\varrho},v) + (u_{\varrho},v) = \lambda_{\varrho} (u_{\varrho},v)
\qquad\forall~v\in\widetilde{V}_{\varrho}(\Gamma).
\end{eqnarray}
We perform numerical simulations for \eqref{eig-dis-gamma} and present the scalings
of $\lambda_{\varrho,{\rm max}}$ (with respect to the discretizations $\varrho$) 
in Figure \ref{fig-inverse} for different sizes of atomic spheres. 

\begin{figure}
	\centering
	\includegraphics[width=8.0cm]{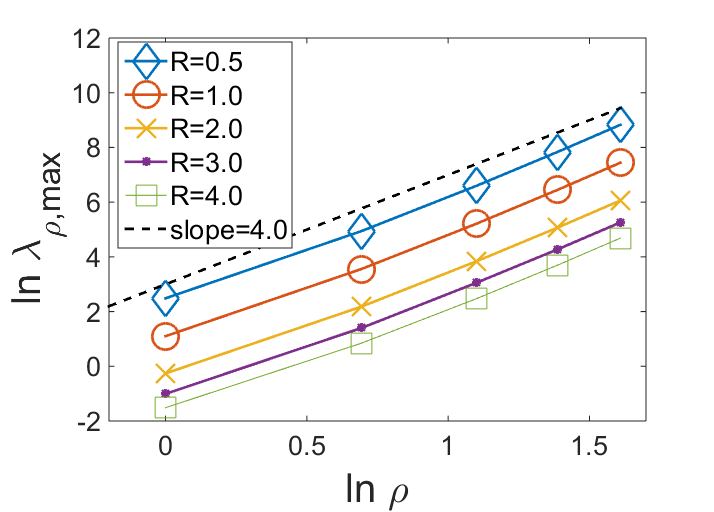}
	\caption{Scalings of the largest eigenvalue of the operator $(-\Delta_{S^2}+1)$ restricted on $\widetilde{V}_{\varrho}(\Gamma)$.}
	\label{fig-inverse}
\end{figure}	

We observe from the numerics that $\lambda_{\varrho,{\rm max}}
= C_R\varrho^4$
for all different radii $R$,  which together with the fact
\begin{eqnarray*}\label{eq-H1-eigenvalue}
	\|u\|^2_{H^1(\Gamma)}\lesssim 
	\lambda_{\varrho,{\rm max}}\|u\|^2_{L^2(\Gamma)}\quad\forall u\in\widetilde{V}_{\varrho}(\Gamma),
\end{eqnarray*}
implies
\begin{eqnarray}\label{eq-inverse-H1}
	\|u\|_{H^1(\Gamma)}\lesssim \varrho^{2}\|u\|_{L^2(\Gamma)}\quad\forall u\in\widetilde{V}_{\varrho}(\Gamma),
\end{eqnarray}
which supports the inverse estimate assumption \eqref{eq-inverse-H1-assumption}.

We then use the ``interpolation" between two spaces $L^2(\Gamma)$ and $H^1(\Gamma)$. For any $u\in L^2(\Gamma) $ and $t>0$, define 
\begin{eqnarray}\label{def-K}
 K(t,u)=\inf_{v\in H^1(\Gamma)}(\|u-v\|_{L^2(\Gamma)}+t\|v\|_{H^1(\Gamma)})
\end{eqnarray}
and the $H^{\frac{1}{2}}$-norm through ``interpolation" \cite{brenner08}:
$\displaystyle\|u\|_{H^\frac{1}{2}(\Gamma)}=\left(\int_{0}^{\infty}\frac{K^2(t,u)}{t^2} dt \right)^{\frac{1}{2}}$. 
For any $0<\alpha<1$, we have 
\begin{eqnarray}\label{eq-interpolation}
\|u\|_{H^\frac{1}{2}(\Gamma)} \lesssim \|u\|_{L^2(\Gamma)}^{\frac{\alpha}{2}} \|u\|_{H^1(\Gamma)}^{1-\frac{\alpha}{2}}.
\end{eqnarray}
To see \eqref{eq-interpolation}, we have 
$K(t,u)\leq t\|u\|_{H^1(\Gamma)}$ by taking $v=u$ in \eqref{def-K} and
$K(t,u)\leq \|u\|_{L^2(\Gamma)}$ by choosing $v=0$. 
Using these two inequalities, we can derive
\begin{eqnarray*}
	\|u\|^2_{H^{\frac{1}{2}}(\Gamma)}
	&=& \int_{0}^{1}\frac{K^2(t,u)}{t^2} dt + \int_{1}^{\infty}\frac{K^2(t,u)}{t^2} dt   \\[1ex]
	&\leq& \|u\|_{L^2(\Gamma)}^{\alpha}\left(\int_{0}^{1}\frac{K^{2-\alpha}(t,u)}{t^2} dt + \int_{1}^{\infty}\frac{K^{2-\alpha}(t,u)}{t^2} dt \right)  \\[1ex]
	&\leq& \|u\|_{L^2(\Gamma)}^{\alpha}\left(\int_{0}^{1}\frac{(t\|u\|_{H^1(\Gamma)})^{2-\alpha}}{t^2} dt + \int_{1}^{\infty}\frac{\|u\|_{L^2(\Gamma)}^{2-\alpha}}{t^2} dt \right)  \\[1ex]
	&\leq & C_{\alpha}\|u\|_{L^2(\Gamma)}^{\alpha}\|u\|_{H^1(\Gamma)}^{2-\alpha}\qquad 
	\forall~0<\alpha<1.
\end{eqnarray*}

Combining \eqref{eq-inverse-H1} and \eqref{eq-interpolation}, we can obtain the inverse estimate \eqref{inverse-estimate} .

\bibliographystyle{plain}
\bibliography{ref}

\end{document}